\documentclass[oneside,english]{amsart}
\usepackage[T1]{fontenc}
\usepackage[latin9]{inputenc}
\usepackage{geometry}
\usepackage{amstext}
\usepackage{amsthm}
\usepackage{amssymb}

\makeatletter
\numberwithin{equation}{section}
\numberwithin{figure}{section}
 \theoremstyle{definition}
 \newtheorem*{defn*}{\protect\definitionname}
\theoremstyle{plain}
\newtheorem{thm}{\protect\theoremname}[section]
  \theoremstyle{remark}
  \newtheorem{notation}[thm]{\protect\notationname}
  \theoremstyle{remark}
  \newtheorem{rem}[thm]{\protect\remarkname}
  \theoremstyle{plain}
  \newtheorem{prop}[thm]{\protect\propositionname}
  \theoremstyle{definition}
  \newtheorem{example}[thm]{\protect\examplename}
  \theoremstyle{plain}
  \newtheorem{lem}[thm]{\protect\lemmaname}
  \theoremstyle{definition}
  \newtheorem{defn}[thm]{\protect\definitionname}
  \theoremstyle{plain}
  \newtheorem{question}[thm]{\protect\questionname}

\usepackage[all]{xy}
\usepackage[backref=page,colorlinks, linktocpage, urlcolor=blue, citecolor = blue, linkcolor = blue]{hyperref}
\usepackage{graphicx}
\usepackage{tikz}
\usepackage{tikz-cd}

\makeatother

\usepackage{babel}
  \providecommand{\definitionname}{Definition}
  \providecommand{\examplename}{Example}
  \providecommand{\lemmaname}{Lemma}
  \providecommand{\notationname}{Notation}
  \providecommand{\propositionname}{Proposition}
  \providecommand{\questionname}{Question}
  \providecommand{\remarkname}{Remark}
\providecommand{\theoremname}{Theorem}

\begin{document}

\author{Adrien Dubouloz}

\address{IMB UMR5584, CNRS, Univ. Bourgogne Franche-Comté, F-21000 Dijon,
France.}

\email{adrien.dubouloz@u-bourgogne.fr}

\author{Isac Hedén}

\address{Research Institute for Mathematical Sciences, Kyoto University, Kyoto
606-8502, Japan}

\email{Isac.Heden@kurims.kyoto-u.ac.jp}

\author{Takashi Kishimoto }

\address{Department of Mathematics, Faculty of Science, Saitama University,
Saitama 338-8570, Japan }

\email{tkishimo@rimath.saitama-u.ac.jp}

\thanks{This work was partially funded by Grant-in-Aid for JSPS Fellows Number
15F15751and Grant-in-Aid for Scientific Research of JSPS No. 15K04805.
The research was done during visits of the first and second authors
at Saitama University, and during visits of the third author at the
Institut de Mathématiques de Bourgogne. The authors thank these institutions
for their generous supports and the excellent working conditions offered.}

\subjclass[2000]{14R20; 14R25; 14R05; 14L30, 14D06}

\title{Equivariant extensions of $\mathbb{G}_{a}$-torsors over punctured
surfaces}
\begin{abstract}
Motivated by the study of the structure of algebraic actions the additive
group on affine threefolds $X$, we consider a special class of such
varieties whose algebraic quotient morphisms $X\rightarrow X/\!/\mathbb{G}_{a}$
restrict to principal homogeneous bundles over the complement of a
smooth point of the quotient. We establish basic general properties
of these varieties and construct families of examples illustrating
their rich geometry. In particular, we give a complete classification
of a natural subclass consisting of threefolds $X$ endowed with proper
$\mathbb{G}_{a}$-actions, whose algebraic quotient morphisms $\pi:X\rightarrow X/\!/\mathbb{G}_{a}$
are surjective with only isolated degenerate fibers, all isomorphic
to the affine plane $\mathbb{A}^{2}$ when equipped with their reduced
structures. 
\end{abstract}

\maketitle

\section*{Introduction}

Algebraic actions of the complex additive group $\mathbb{G}_{a}=\mathbb{G}_{a,\mathbb{C}}$
on normal complex affine surfaces $S$ are essentially fully understood:
the ring of invariants $\mathcal{O}(S){}^{\mathbb{G}_{a,\mathbb{C}}}$
is a finitely generated algebra whose spectrum is a smooth affine
curve $C=S/\!/\mathbb{G}_{a}$, and the inclusion $\mathcal{O}(S)^{\mathbb{G}_{a}}\subset\mathcal{O}(S)$
defines a surjective morphism $\pi:S\rightarrow C$ whose general
fibers coincide with general orbits of the action, hence are isomorphic
to the affine line $\mathbb{A}^{1}$ on which $\mathbb{G}_{a}$ acts
by translations. The degenerate fibers of such $\mathbb{A}^{1}$-fibrations
are known to consist of finite disjoint unions of smooth affine curves
isomorphic to $\mathbb{A}^{1}$ when equipped with their reduced structure.
A complete description of isomorphism classes of germs of invariant
open neighborhoods of irreducible components of such fibers was established
by Fieseler \cite{Fie94}. 

In contrast, very little is known so far about the structure of $\mathbb{G}_{a}$-actions
on complex normal affine threefolds. For such a threefold $X$, the
ring of invariants $\mathcal{O}(X){}^{\mathbb{G}_{a}}$ is again finitely
generated \cite{Na59} and the morphism $\pi:X\rightarrow S$ induced
by the inclusion $\mathcal{O}(X)^{\mathbb{G}_{a}}\subset\mathcal{O}(X)$
is an $\mathbb{A}^{1}$-fibration over a normal affine surface $S$.
But in general, $\pi$ is neither surjective nor equidimensional.
Furthermore, it can have degenerate fibers over closed subsets of
pure codimension $1$ as well as of codimension $2$. All of these
possible degeneration are illustrated by the following example:

The restriction of the projection $\mathrm{pr}_{x,y}$ to the smooth
threefold $X=\{x^{2}(x-1)v+yu^{2}-x=0\}$ in $\mathbb{A}^{4}$ is
an $\mathbb{A}^{1}$-fibration $\pi:X\rightarrow\mathbb{A}^{2}$ which
coincides with the algebraic quotient morphism of the $\mathbb{G}_{a}$-action
on $X$ associated to the locally nilpotent derivation $\partial=x^{2}(x-1)\partial_{u}-2yu\partial_{v}$
of its coordinate ring. The restriction of $\pi$ over the principal
open subset $x^{2}(x-1)\neq0$ of $\mathbb{A}^{2}$ is a trivial principal
$\mathbb{G}_{a}$-bundle, but the fibers of $\pi$ over the points
$(1,0)$ and $(0,0)$ are respectively empty and isomorphic to $\mathbb{A}^{2}$.
Furthermore, for every $y_{0}\neq0$, the inverse images under $\pi$
of the points $(0,y_{0})$ and $(1,y_{0})$ are respectively isomorphic
to $\mathbb{A}^{1}$ but with multiplicity $2$, and to the disjoint
union of two reduced copies of $\mathbb{A}^{1}$.

Partial results concerning the structure of one-dimensional degenerate
fibers of $\mathbb{G}_{a}$-quotient $\mathbb{A}^{1}$-fibrations
were obtained by Gurjar-Masuda-Miyanishi \cite{GMM12}. In the present
article, as a step towards the understanding of the structure of two-dimensional
degenerate fibers, we consider a particular type of non equidimensional
surjective $\mathbb{G}_{a}$-quotient $\mathbb{A}^{1}$-fibrations
$\pi:X\rightarrow S$ which have the property that they restrict to
$\mathbb{G}_{a}$-torsors\footnote{sometimes also referred to as Zariski locally trivial principal $\mathbb{G}_{a}$-bundles}
 over the complement of a finite set of smooth points in $S$. These
are simpler than the general case illustrated in the previous example
since they do not admit additional degeneration of their fibers over
curves in $S$ passing through the given points. The local and global
study of some classes of such fibrations was initiated by the second
author \cite{He15}. He constructed in particular many examples of
$\mathbb{G}_{a}$-quotient $\mathbb{A}^{1}$-fibrations on smooth
affine threefolds $X$ with image $\mathbb{A}^{2}$ whose restrictions
over the complement of the origin are isomorphic to the geometric
quotient $\mathrm{SL}_{2}\rightarrow\mathrm{SL}_{2}/\mathbb{G}_{a}$
of $\mathrm{SL}_{2}$ by the action of unitary upper triangular matrices. 

One of the simplest examples of this type is the smooth threefold
$X_{0}\subset\mathbb{A}_{x,y,p,q,r}^{5}$ defined by the equations
\[
X_{0}:\qquad\begin{cases}
xr-yq & =0\\
yp-x(q-1) & =0\\
pr-q(q-1) & =0
\end{cases}
\]
and equipped with the $\mathbb{G}_{a}$-action associated to the locally
nilpotent $\mathbb{C}[x,y]$-derivation $x^{2}\partial_{p}+xy\partial_{q}+y^{2}\partial_{r}$
of its coordinate ring. The equivariant open embedding $\mathrm{SL}_{2}=\{xv-yu=1\}\hookrightarrow X_{0}$
is given by $(x,y,u,v)\mapsto(x,y,xu,xv,yv)$. The $\mathbb{G}_{a}$-quotient
morphism coincides with the surjective $\mathbb{A}^{1}$-fibration
$\pi_{0}:\mathrm{pr}_{x,y}:X_{0}\rightarrow\mathbb{A}^{2}$. Its restriction
over $\mathbb{A}^{2}\setminus\{(0,0)\}$ is isomorphic to the quotient
morphism $\mathrm{SL}_{2}\rightarrow\mathrm{SL}_{2}/\mathbb{G}_{a}$,
while its fiber over $(0,0)$ is the smooth quadric $\{pr-q(q-1)=0\}\subset\mathbb{A}_{p,q,r}^{3}$,
isomorphic to the quotient $\mathrm{SL}_{2}/\mathbb{G}_{m}$ of $\mathrm{SL}_{2}$
by the action of its diagonal torus (see Example \ref{subsec:Homogeneous-Ga-torsors}).
A noteworthy property of this example is that the $\mathbb{G}_{a}$-quotient
morphism $\pi:X_{0}\rightarrow\mathbb{A}^{2}$ factors through a locally
trivial $\mathbb{A}^{1}$-bundle $\rho:X_{0}\rightarrow\tilde{\mathbb{A}}^{2}$
over the the blow-up $\tau:\tilde{\mathbb{A}}^{2}\rightarrow\mathbb{A}^{2}$
of the origin. 

It is a general fact that every irreducible component of a degenerate
fiber of pure codimension one of a $\mathbb{G}_{a}$-quotient $\mathbb{A}^{1}$-fibration
$\pi:X\rightarrow S$ on a smooth affine threefold is an $\mathbb{A}^{1}$-uniruled
affine surface (see Proposition \ref{prop:Uniruled-central-fiber}).
We do not know whether every $\mathbb{A}^{1}$-uniruled surface can
be realized as an irreducible component of the degenerate fiber of
a $\mathbb{G}_{a}$-extension. But besides the smooth affine quadric
$\mathrm{SL}_{2}/\mathbb{G}_{m}$ appearing in the previous example,
the following one confirms that the affine plane $\mathbb{A}^{2}$
can also be realized (see also Examples \ref{exa:2-components-fiber}
and \ref{exa:A1ruled-fiber} for other types of surfaces that can
be realized): Let $X_{1}\subset\mathbb{A}_{x,y,z_{1}z_{2},w}^{5}$
be the smooth affine threefold defined by the equations 
\[
X_{1}:\qquad\begin{cases}
xw-y(yz_{1}+1) & =0\\
xz_{2}-z_{1}(yz_{1}+1) & =0\\
z_{1}w-yz_{2} & =0,
\end{cases}
\]
equipped with the $\mathbb{G}_{a}$-action associated to the locally
nilpotent $\mathbb{C}[x,y]$-derivation $x\partial_{z_{1}}+(2yz_{1}+1)\partial_{z_{2}}+y^{2}\partial_{w}$
of its coordinate ring. The morphism $\mathrm{SL}_{2}\hookrightarrow X_{1}$
given by $(x,y,u,v)\mapsto(x,y,u,uv,yv)$ isequivariant open embedding.
The $\mathbb{G}_{a}$-quotient morphism coincides with the surjective
$\mathbb{A}^{1}$-fibration $\pi_{1}=\mathrm{pr}_{x,y}:X_{1}\rightarrow\mathbb{A}^{2}$,
whose fiber over the origin is the affine plane $\mathbb{A}^{2}=\mathrm{Spec}(\mathbb{C}[z_{2},w])$
and whose restriction over $\mathbb{A}^{2}\setminus\{(0,0)\}$ is
again isomorphic to the quotient morphism $\mathrm{SL}_{2}\rightarrow\mathrm{SL}_{2}/\mathbb{G}_{a}$.
A special additional feature is that the $\mathbb{G}_{a}$-action
on $X_{1}$ extending that on $\mathrm{SL}_{2}$ is not only fixed
point free but actually \emph{proper}: its geometric quotient $X_{1}/\mathbb{G}_{a}$
is separated. One can indeed check that $X_{1}/\mathbb{G}_{a}$ is
isomorphic to the complement $\tilde{\mathbb{A}}^{2}\setminus\{o_{1}\}$
of a point $o_{1}$ supported on the exceptional divisor $E$ of the
blow-up $\tilde{\mathbb{A}}^{2}$ of $\mathbb{A}^{2}$ at the origin
(see Example \ref{exa:ProperExtension}). \\

Relaxing the hypothesis that the $\mathbb{A}^{1}$-fibration $\pi:X\rightarrow S$
arises as the quotient of a $\mathbb{G}_{a}$-action on an affine
threefold $X$ to consider the broader problem of describing the geometry
of degeneration of $\mathbb{A}^{1}$-fibrations over irreducible closed
subsets of pure codimension two of their base, we are led to the following
more general notion:
\begin{defn*}
Let $(S,o)$ be a pair consisting of a normal separated $2$-dimensional
scheme $S$ essentially of finite type over a field $k$ of characteristic
zero and of a closed point $o$ contained in the smooth locus of $S$.
A $\mathbb{G}_{a}$-\emph{extension} of a $\mathbb{G}_{a}$-torsor
$\rho:P\rightarrow S\setminus\{o\}$ is a $\mathbb{G}_{a}$-equivariant
open embedding $j:P\hookrightarrow X$ into an integral scheme $X$
equipped with a surjective morphism $\pi:X\rightarrow S$ of finite
type and a $\mathbb{G}_{a,S}$-action, such that the commutative diagram
\[\begin{tikzcd}    P \arrow[r,hook,"j"] \arrow[d,"\rho"'] & X \arrow[d,"\pi"] \\ S\setminus\{o\} \arrow[r,hook] & S \end{tikzcd}\]
is cartesian.
\end{defn*}
The examples $X_{0}$ and $X_{1}$ above provide motivation to study
the following natural classes of $\mathbb{G}_{a}$-extensions $\pi:X\rightarrow S$
of a $\mathbb{G}_{a}$-torsor $\rho:P\rightarrow S\setminus\{o\}$,
which are arguably the simplest possible types of $\mathbb{G}_{a}$-extensions
from the viewpoints of their global geometry and of the properties
of their $\mathbb{G}_{a}$-actions:

- (Type I) Extensions for which $\pi$ factors through a locally trivial
$\mathbb{A}^{1}$-bundle over the blow-up $\tau:\tilde{S}\rightarrow S$
of the point $o$, the fiber $\pi^{-1}(o)$ being then the total space
of a locally trivial $\mathbb{A}^{1}$-bundle over the exceptional
divisor of $\tau$. 

- (Type II) Extensions for which $\pi^{-1}(o)_{\mathrm{red}}$ is
isomorphic to the affine plane $\mathbb{A}_{\kappa}^{2}$ over the
residue field $\kappa$ of $S$ at $o$, $X$ is smooth along $\pi^{-1}(o)$
and the $\mathbb{G}_{a,S}$-action on $X$ is proper. \\

The first main result of this article, Proposition \ref{prop:Blowup-Extension}
and Theorem \ref{prop:Main-Ext-Prop}, is a complete description of
$\mathbb{G}_{a}$-extensions of Type I together with an effective
characterization of which among them have the additional property
that the morphism $\pi:X\rightarrow S$ is affine. Our second main
result, Theorem \ref{thm:Main-TheoremA2}, consists of a classification
of $\mathbb{G}_{a}$-extensions of Type II, under the additional assumption
that the morphism $\pi:X\rightarrow S$ is quasi-projective. More
precisely, given a $\mathbb{G}_{a}$-torsor $\rho:P\rightarrow S\setminus\{o\}$
and a $\mathbb{G}_{a}$-extension $\pi:X\rightarrow S$ with proper
$\mathbb{G}_{a,S}$-action and reduced fiber $\pi^{-1}(o)_{\mathrm{red}}$
isomorphic to $\mathbb{A}_{\kappa}^{2}$, we establish that the possible
geometric quotients $S'=X/\mathbb{G}_{a}$ belong to a very special
class of surfaces isomorphic to open subsets of blow-ups of $S$ with
centers over $o$ which we fully describe in $\S$ \ref{subsec:ExtensionFamilies}.
We show conversely that every such surface is indeed the geometric
quotient of a $\mathbb{G}_{a}$-extension of $\rho:P\rightarrow S\setminus\{o\}$
with the desired properties. 

In a second step, we tackle the question of existence of $\mathbb{G}_{a}$-extensions
$\pi:X\rightarrow S$ of Type II for which the structure morphism
$\pi$ is not only quasi-projective but affine. Our method to produce
extensions with this property is inspired by the observation that
the threefolds $X_{0}$ and $X_{1}$ above are not only birational
to each other due to the property that they both contain $\mathrm{SL}_{2}$
as open subset, but in fact that the birational morphism 
\[
\eta:X_{1}\rightarrow X_{0},\quad(x,y,z_{1},z_{2},w)\mapsto(x,y,p,q,r)=(x,y,xz_{1},yz_{1}+1,w)
\]
expresses $X_{1}$ as a $\mathbb{G}_{a}$-equivariant affine modification
of $X_{0}$ in the sense of Kaliman and Zaidenberg \cite{KaZa99}.
This suggests that extensions of Type II for which $X$ is affine
over $S$ could be obtained as equivariant affine modification in
a suitable generalized sense from extensions of Type I with the same
property. Using this technique, we are able to show in Theorem \ref{thm:Affine-Ga-extensions-A2}
that for each possible geometric quotient $S'$ above, there exist
$\mathbb{G}_{a}$-extensions $\pi:X\rightarrow S$ of $\rho:P\rightarrow S\setminus\{o\}$
with geometric quotient $X/\mathbb{G}_{a}=S'$ such that $\pi$ is
an affine morphism. 

As an application towards the initial question of the structure $\mathbb{G}_{a}$-quotient
$\mathbb{A}^{1}$-fibrations on affine threefolds, we in particular
derive from this construction the existence of uncountably many pairwise
non isomorphic smooth affine threefolds $X$ endowed with proper $\mathbb{G}_{a}$-actions,
containing $\mathrm{SL}_{2}$ as an invariant open subset with complement
$\mathbb{A}^{2}$, whose geometric quotients are smooth quasi-projective
surfaces which are not quasi-affine, and whose algebraic quotients
are all isomorphic to $\mathbb{A}^{2}$. \\

The scheme of the article is the following. The first section begins
with a review of general properties of $\mathbb{G}_{a}$-extensions.
We then set up the basic tools which will be used through all the
article: locally trivial $\mathbb{A}^{1}$-bundles with additive group
actions and equivariant affine birational morphisms between these.
In section two, we study $\mathbb{G}_{a}$-extensions of Type I. The
last section is devoted to the classification of quasi-projective
$\mathbb{G}_{a}$-extensions of Type II. 

\tableofcontents{}

\section{Preliminaries }
\begin{notation}
In the rest of the article, the term \emph{surface }refers to a normal
separated $2$-dimensional scheme essentially of finite type over
a field $k$ of characteristic zero. A \emph{punctured surface} $S_{*}=S\setminus\{o\}$
is the complement of a closed point $o$ contained in the smooth locus
of a surface $S$. We denote by $\kappa$ the residue field of $S$
at $o$. 
\end{notation}

\begin{rem}
We do not require that the residue field $\kappa$ of $S$ at $o$
is an algebraic extension of $k$. For instance, $S$ can very well
be the spectrum of the local ring $\mathcal{O}_{X,Z}$ of an arbitrary
smooth $k$-variety $X$ at an irreducible closed subvariety $Z$
of codimension two in $X$ and $o$ its unique closed point, in which
case the residue field $\kappa$ is isomorphic to the field of rational
functions on $Z$. 
\end{rem}

In this section, we first review basic geometric properties of equivariant
extensions of $\mathbb{G}_{a}$-torsors over punctured surfaces. We
then collect various technical results on additive group actions on
affine-linear bundles of rank one and their behavior under equivariant
affine modifications. 

\subsection{Equivariant extensions of $\mathbb{G}_{a}$-torsors }

A $\mathbb{G}_{a}$-torsor over punctured surface $S_{*}=S\setminus\{o\}$
is an $S_{*}$-scheme $\rho:P\rightarrow S_{*}$ equipped with a $\mathbb{G}_{a}$-action
$\mu:\mathbb{G}_{a,S_{*}}\times_{S_{*}}P\rightarrow P$ for which
there exists a Zariski open cover $f:Y\rightarrow S_{*}$ of $S_{*}$
such that $P\times_{S_{*}}Y$ is equivariantly isomorphic to $\mathbb{G}_{a,Y}$
acting on itself by translations. In the present article, we primarily
focus on $\mathbb{G}_{a}$-torsors $\rho:P\rightarrow S_{*}$ whose
restrictions $P\times_{S_{*}}U\rightarrow U\setminus\{o\}$ over every
Zariski open neighborhood $U$ of $o$ in $S$ are nontrivial. Since
in this case the total space of $P$ is affine over $S$ (see e.g.
\cite[Proposition 1.2]{DF14} whose proof carries over verbatim to
our more general situation), it follows that for every $\mathbb{G}_{a}$-extension
$j:P\hookrightarrow X$ the fiber $\pi^{-1}(o)\subset X$ of the surjective
morphism $\pi:X\rightarrow S$ has pure codimension one in $X$. Two
important families of examples of non trivial normal $\mathbb{G}_{a}$-extensions
$j:\mathrm{SL}_{2}\rightarrow X$ of the $\mathbb{G}_{a}$-torsor
$\rho:\mathrm{SL}_{2}\rightarrow\mathrm{SL}_{2}/\mathbb{G}_{a}\simeq\mathbb{A}^{2}\setminus\{(0,0)\}$,
where $\mathbb{G}_{a}$ acts on $\mathrm{SL}_{2}$ via left multiplication
by upper triangular unipotent matrices, were constructed in \cite[Section 5 and 6]{He15}.
Various other extensions were obtained from these by performing suitable
equivariant affine modifications. One can observe that for all these
extensions, the fiber $\pi^{-1}(\{(0,0)\})$ is an $\mathbb{A}^{1}$-ruled
surface, a property which is a consequence of the following more general
fact: 
\begin{prop}
\label{prop:Uniruled-central-fiber} Let $\rho:P\rightarrow S_{*}$
be a non trivial $\mathbb{G}_{a}$-torsor over the punctured spectrum
$S\setminus\{o\}$ of a regular local ring of dimension $2$ over
an algebraically closed field $k$ and with residue field $\kappa(o)=k$,
and let $\pi:X\rightarrow S$ be a $\mathbb{G}_{a}$-extension of
$P$. If $X$ is smooth along $\pi^{-1}(o)$, then every irreducible
component $F$ of $\pi^{-1}(o)_{\mathrm{red}}$ is a uniruled surface.
Furthermore, if $X$ is affine then $F$ is $\mathbb{A}^{1}$-uniruled,
hence $\mathbb{A}^{1}$-ruled when it is normal. 
\end{prop}

\begin{proof}
Since $\pi^{-1}(o)$ has pure codimension one in $X$ and $X$ is
smooth along $\pi^{-1}(o)$, every irreducible component of $\pi^{-1}(o)$
is a $\mathbb{G}_{a}$-invariant Cartier divisor on $X$. The complement
$X'$ in $X$ of all but one irreducible component of $\pi^{-1}(o)$
is thus again a $\mathbb{G}_{a}$-extension of $P$, and we may therefore
assume without loss of generality that $F=\pi^{-1}(o)_{\mathrm{red}}$
is irreducible. Let $x\in F$ be a closed point in the regular locus
of $F$. Since $F$ and $X$ are smooth at $x$ and $X$ is connected,
there exists a curve $C\subset X$, smooth at $x$ and intersecting
$F$ transversally at $x$. The image $\pi(C)$ of $C$ is a curve
on $S$ passing through $o$, and the closure $B$ of $\pi^{-1}(\pi(C)\cap S_{*})$
in $X$ is a surface containing $C$. Since $\rho:P\rightarrow S_{*}$
is a $\mathbb{G}_{a}$-torsor, the restriction of $\pi$ to $B\cap P$
is a trivial $\mathbb{G}_{a}$-torsor over the affine curve $\pi(C)$.
So $\pi\mid_{B}:B\rightarrow\pi(C)$ is an $\mathbb{A}^{1}$-fibration.
Let $\nu:\tilde{C}\rightarrow\pi(C)$ be the normalization of $\pi(C)$.
Then $\pi\mid_{B}$ lifts to an $\mathbb{A}^{1}$-fibration $\theta:\tilde{B}\rightarrow\tilde{C}$
on the normalization $\tilde{B}$ of $B$. The fiber of $\theta$
over every point in $\nu^{-1}(o)$ is a union of rational curves.
Since the normalization morphism $\mu:\tilde{B}\rightarrow B$ is
surjective, one of the irreducible components of $\nu^{-1}(o)$ is
mapped by $\mu$ onto a rational curve in $F$ passing through $x$.
This shows that for every smooth closed point $x$ of $F$, there
exists a non constant rational map $h:\mathbb{P}^{1}\dashrightarrow F$
such that $x\in h(\mathbb{P}^{1})$. Thus $F$ is uniruled. If $X$
is in addition affine, then $B$ and $\tilde{B}$ are affine surfaces,
and the fibers of the $\mathbb{A}^{1}$-fibration $\theta:\tilde{B}\rightarrow\tilde{C}$
consist of disjoint union of curves isomorphic to $\mathbb{A}^{1}$
when equipped with their reduced structure. This implies that $F$
is not only uniruled but actually $\mathbb{A}^{1}$-uniruled. 
\end{proof}
\begin{example}
\label{exa:2-components-fiber}Let $X$ be the smooth affine threefold
in $\mathbb{A}^{2}\times\mathbb{A}^{4}=\mathrm{Spec}(k[x,y][c,d,e,f])$
defined by the equations 
\[
\begin{cases}
xd-y(c+1) & =0\\
xc^{2}-y^{2}e & =0\\
yf-c(c+1) & =0\\
xf^{2}-(c+1)^{2}e & =0\\
de-cf & =0
\end{cases}
\]
equipped with the $\mathbb{G}_{a}$-action induced by the locally
nilpotent $k[x,y]$-derivation 
\[
xy\partial_{c}+y^{2}\partial_{d}+x(2c+1)\partial_{f}+(2x^{2}f-2xye)\partial_{e}
\]
of its coordinate ring. The morphism $j:\mathrm{SL}_{2}=\{xv-yu=1\}\rightarrow X$
defined by $(x,y,u,v)\mapsto(x,y,yu,yv,xu^{2},xuv)$ is an open embedding
of $\mathrm{SL}_{2}$ in $X$ as the complement of the fiber over
$o=(0,0)$ of the projection $\pi=\mathrm{pr}_{x,y}:X\rightarrow\mathbb{A}^{2}$.
So $j:\mathrm{SL}_{2}\rightarrow X$ is an affine $\mathbb{G}_{a}$-extension
of the $\mathbb{G}_{a}$-torsor $\rho:\mathrm{SL}_{2}\rightarrow\mathrm{SL}_{2}/\mathbb{G}_{a}=\mathbb{A}^{2}\setminus\{o\}$,
for which $\pi^{-1}(o)$ consists of the disjoint union of two copies
$D_{1}=\{x=y=c=0\}\simeq\mathrm{Spec}(k[d,f])$ and $D_{2}=\{x=y=c+1=0\}\simeq\mathrm{Spec}(k[d,e])$
of $\mathbb{A}^{2}$. Note that the induced $\mathbb{G}_{a}$-action
on each of these is the trivial one. 
\end{example}

\begin{example}
\label{exa:A1ruled-fiber}Let $X$ be the affine $\mathbb{G}_{a}$-extension
constructed in the previous example and let $C\subset D_{1}$ be any
smooth affine curve. Let $\tau:\tilde{X}\rightarrow X$ be the blow-up
of $X$ along $C$, let $i:X'\hookrightarrow\tilde{X}$ be the open
immersion of the complement of the proper transform of $D_{1}\cup D_{2}$
in $\tilde{X}$ and let $\pi'=\pi\circ\tau\circ i:X'\rightarrow\mathbb{A}^{2}$.
Since $C$ and $D_{1}\cup D_{2}$ are $\mathbb{G}_{a}$-invariant,
the $\mathbb{G}_{a}$-action on $X$ lifts to a $\mathbb{G}_{a}$-action
on $\tilde{X}$ which restricts in turn to $X'$. By construction,
$\pi'$ is surjective, with fiber ${\pi'}^{-1}(o)$ isomorphic to
$C\times\mathbb{A}^{1}$ and $\tau\circ i:X'\rightarrow X$ restricts
to an equivariant isomorphism between $X'\setminus{\pi'}^{-1}(o)$
and $X\setminus\pi^{-1}(o)\simeq\mathrm{SL}_{2}$. So $\pi':X'\rightarrow\mathbb{A}^{2}$
is a $\mathbb{G}_{a}$-extension of the $\mathbb{G}_{a}$-torsor $\rho:\mathrm{SL}_{2}\rightarrow\mathrm{SL}_{2}/\mathbb{G}_{a}=\mathbb{A}^{2}\setminus\{o\}$. 
\end{example}

\subsection{Recollection on affine-linear bundles}

Affine-linear bundles of rank one over a scheme are natural generalization
of $\mathbb{G}_{a}$-torsors. To fix the notation, we briefly recall
their basic definitions and properties 

By a line bundle on a scheme $S$, we mean the relative spectrum $p:M=\mathrm{Spec}(\mathrm{Sym}^{\cdot}\mathcal{M}^{\vee})\rightarrow S$
of the symmetric algebra of the dual of an invertible sheaf of $\mathcal{O}_{S}$-module
$\mathcal{M}$. Such a line bundle $M$ can be viewed as a locally
constant group scheme over $S$ for the group law $m:M\times_{S}M\rightarrow M$
whose co-morphism 
\[
m^{\sharp}:\mathrm{Sym}^{\cdot}\mathcal{M}^{\vee}\rightarrow\mathrm{Sym}^{\cdot}\mathcal{M}^{\vee}\otimes\mathrm{Sym}^{\cdot}\mathcal{M}^{\vee}\simeq\mathrm{Sym}^{\cdot}(\mathcal{M}^{\vee}\oplus\mathcal{M}^{\vee})
\]
is induced by the diagonal homomorphism $\mathcal{M}^{\vee}\rightarrow\mathcal{M}^{\vee}\oplus\mathcal{M}^{\vee}$.
An $M$-\emph{torsor} is then an $S$-scheme $\theta:W\rightarrow S$
equipped with an action $\mu:M\times_{S}W\rightarrow W$ which is
Zariski locally over $S$ isomorphic to $M$ acting on itself by translations. 

This is the case precisely when there exists a Zariski open cover
$f:Y\rightarrow S$ and an $\mathcal{O}_{Y}$-algebra isomorphism
$\psi:f^{*}\mathcal{A}\rightarrow\mathrm{Sym}^{\cdot}f^{*}\mathcal{M}^{\vee}$
such that over $Y'=Y\times_{S}Y$ the automorphism $\mathrm{p}_{1}^{*}\psi\circ\mathrm{p}_{2}^{*}\psi^{-1}:\mathrm{Sym}^{\cdot}\mathcal{M}_{Y'}^{\vee}\rightarrow\mathrm{Sym}^{\cdot}\mathcal{M}_{Y'}^{\vee}$
of the symmetric algebra of $\mathcal{M}_{Y'}^{\vee}=\mathrm{p}_{2}^{*}f^{*}\mathcal{M}^{\vee}=\mathrm{p}_{1}^{*}f^{*}\mathcal{M}^{\vee}$
is \emph{affine-linear}, i.e. induced by an $\mathcal{O}_{Y'}$-module
homomorphism $\mathcal{M}_{Y'}^{\vee}\rightarrow\mathrm{Sym}^{\cdot}\mathcal{M}_{Y'}^{\vee}$
of the form 
\begin{equation}
\beta\oplus\mathrm{id}:\mathcal{M}_{Y'}^{\vee}\rightarrow\mathcal{O}_{Y'}\oplus\mathcal{M}_{Y'}^{\vee}\hookrightarrow\bigoplus_{n\geq0}(\mathcal{M}_{Y'}^{\vee})^{\otimes n}=\mathrm{Sym}^{\cdot}\mathcal{M}_{Y'}^{\vee}\label{eq:AffLine-GluingBis}
\end{equation}
for some $\beta\in\mathrm{Hom}_{Y'}(\mathcal{M}_{Y'}^{\vee},\mathcal{O}_{Y'})\simeq H^{0}(Y',\mathcal{M}_{Y'})$
which is a \v{C}ech $1$-cocyle with values in $\mathcal{M}$ for
the Zariski open cover $f:Y\rightarrow S$. Standard arguments show
that the isomorphism class of $\theta:W\rightarrow S$ depends only
on the class of $\beta$ in the \v{C}ech cohomology group $\check{H}^{1}(S,\mathcal{M})$,
and one eventually gets a one-to-one correspondence between isomorphism
classes of $M$-torsors over $S$ and elements of the cohomology group
$H^{1}(S,M)=H^{1}(S,\mathcal{M})\simeq\check{H}^{1}(S,\mathcal{M})$
with the zero element corresponding to the trivial torsor $p:M\rightarrow S$. 

It is classical that every locally trivial $\mathbb{A}^{1}$-bundle
$\theta:W\rightarrow S$ over a reduced scheme $S$ can be equipped
with the additional structure of a torsor under a uniquely determined
line bundle $M$ on $S$. The existence of this additional structure
will be frequently used in the sequel, and we now quickly review its
construction (see also e.g. \cite[§ 2.3 and § 2.4]{DuTG05}). Letting
$\mathcal{A}=\theta_{*}\mathcal{O}_{W}$, there exists by definition
a Zariski open cover $f:Y\rightarrow S$ and a quasi-coherent $\mathcal{O}_{Y}$-algebra
isomorphism $\varphi:f^{*}\mathcal{A}\rightarrow\mathcal{O}_{Y}[u]$.
Over $Y'=Y\times_{S}Y$ equipped with the two projections $\mathrm{p}_{1}$
and $\mathrm{p}_{2}$ to $Y$, the $\mathcal{O}_{Y'}$-algebra isomorphism
$\Phi=\mathrm{p}_{1}^{*}\varphi\circ\mathrm{p}_{2}^{*}\varphi^{-1}$
has the form 
\begin{equation}
\Phi:\mathcal{O}_{Y'}[u]\rightarrow\mathcal{O}_{Y'}[u],\;u\mapsto au+b\label{eq:Affine-glueing}
\end{equation}
for some $a\in\Gamma(Y',\mathcal{O}_{Y'}^{*})$ and $b\in\Gamma(Y',\mathcal{O}_{Y'})$
whose pull back over $Y''=Y\times_{S}Y\times_{S}Y$ by the three projections
$\mathrm{p}_{12},\mathrm{p}_{23},\mathrm{p}_{13}:Y''\rightarrow Y'$
satisfy the cocycle relations $\mathrm{p}_{13}^{*}a=\mathrm{p}_{23}^{*}a\cdot\mathrm{p}_{12}^{*}a$
and $\mathrm{p}_{13}^{*}b=\mathrm{p}_{23}^{*}a\cdot\mathrm{p}_{12}^{*}b+\mathrm{p}_{23}^{*}b$
in $\Gamma(Y'',\mathcal{O}_{Y''}^{*})$ and $\Gamma(Y'',\mathcal{O}_{Y''})$
respectively. The first one says that $a$ is a \v{C}ech $1$-cocycle
with values in $\mathcal{O}_{S}^{*}$ for the cover $f:Y\rightarrow S$,
which thus determines, via the isomorphism $H^{1}(S,\mathcal{O}_{S}^{*})\simeq\mathrm{Pic}(S)$,
a unique invertible sheaf $\mathcal{M}$ on $S$ together with an
$\mathcal{O}_{Y}$-module isomorphism $\alpha:f^{*}\mathcal{M}^{\vee}\rightarrow\mathcal{O}_{Y}$
such that $\mathrm{p}_{1}^{*}\alpha\circ\mathrm{p}_{2}^{*}\alpha^{-1}:\mathcal{O}_{Y'}\rightarrow\mathcal{O}_{Y'}$
is the multiplication by $a$. The second one can be equivalently
reinterpreted as the fact that $\beta=\mathrm{p}_{2}^{*}({}^{t}\alpha)(b)\in\Gamma(Y',\mathcal{M}_{Y'})$
is a \v{C}ech $1$-cocycle with values in $\mathcal{M}$ for the
Zariski open cover $f:Y\rightarrow S$. Letting $\mathrm{Sym}^{\cdot}(\alpha):\mathrm{Sym}^{\cdot}f^{*}\mathcal{M}^{\vee}\rightarrow\mathcal{O}_{Y}[u]$
be the graded $\mathcal{O}_{Y}$-algebra isomorphism induced by $\alpha$,
the isomorphism $\psi=\mathrm{Sym}^{\cdot}(\alpha^{-1})\circ\varphi:f^{*}\mathcal{A}\rightarrow\mathrm{Sym}^{\cdot}f^{*}\mathcal{M}^{\vee}$
has the property that $\mathrm{p}_{1}^{*}\psi\circ\mathrm{p}_{2}^{*}\psi^{-1}$
is affine-linear, induced by the homomorphism $\beta\oplus\mathrm{id}:\mathcal{M}_{Y'}^{\vee}\rightarrow\mathcal{O}_{Y'}\oplus\mathcal{M}_{Y'}^{\vee}$.
So $\theta:W\rightarrow S$ is a torsor under the line bundle $M=\mathrm{Spec}(\mathrm{Sym}^{\cdot}\mathcal{M}^{\vee})$,
with isomorphism class in $H^{1}(S,M)$ equal to the cohomology class
of the cocyle $\beta$. Summing up, we obtain;
\begin{prop}
\label{prop:A1bundle-torsor}Let $\theta:W\rightarrow S$ be a locally
trivial $\mathbb{A}^{1}$-bundle. Then there exists a unique pair
$(M,g)$ consisting of a line bundle $M$ on $S$ and a class $g\in H^{1}(S,M)$
such that $\theta:W\rightarrow S$ is an $M$-torsor with isomorphism
class $g$. 
\end{prop}

\subsection{Additive group actions on affine-linear bundles of rank one }

Given a locally trivial $\mathbb{A}^{1}$-bundle $\theta:W\rightarrow S$,
which we view as an $M$-torsor for a line bundle $M=\mathrm{Spec}(\mathrm{Sym}^{\cdot}\mathcal{M}^{\vee})\rightarrow S$
on $S$, with corresponding action $\mu:M\times_{S}W\rightarrow W$,
every nonzero group scheme homomorphism $\xi:\mathbb{G}_{a,S}\rightarrow M$
induces a nontrivial $\mathbb{G}_{a,S}$-action $\nu=\mu\circ(\xi\times\mathrm{id}):\mathbb{G}_{a,S}\times_{S}W\rightarrow W$
on $W$. A nonzero group scheme homomorphism $\xi:\mathbb{G}_{a,S}=\mathrm{Spec}(\mathcal{O}_{S}[t])\rightarrow M=\mathrm{Spec}(\mathrm{Sym}^{\cdot}\mathcal{M}^{\vee})$
is uniquely determined by a nonzero $\mathcal{O}_{S}$-module homomorphism
$\mathcal{M}^{\vee}\rightarrow\mathcal{O}_{S}$, equivalently by a
nonzero global section $s\in\Gamma(S,\mathcal{M})$. The following
proposition asserts conversely that every nontrivial $\mathbb{G}_{a,S}$-action
on an $M$-torsor $\theta:W\rightarrow S$ uniquely arises from such
a section. 
\begin{prop}
\label{prop:Torsor-Ga-action}$($\cite[Chapter 3]{DuPhD04}$)$ Let
$\theta:W\rightarrow S$ be a torsor under the action $\mu:M\times_{S}W\rightarrow W$
of a line bundle $M=\mathrm{Spec}(\mathrm{Sym}^{\cdot}\mathcal{M}^{\vee})\rightarrow S$
on $S$ and let $\nu:\mathbb{G}_{a,S}\times_{S}W\rightarrow W$ be
a non trivial $\mathbb{G}_{a,S}$-action on $W$. Then there exists
a non zero global section $s\in\Gamma(S,\mathcal{M})$ such that $\nu=\mu\circ(\xi\times\mathrm{id})$
where $\xi:\mathbb{G}_{a,S}\rightarrow M$ is the group scheme homomorphism
induced by $s$. 
\end{prop}

\begin{proof}
Let $\mathcal{A}=\theta_{*}\mathcal{O}_{W}$ and let $f:Y\rightarrow S$
be a Zariski open cover such that there exists an $\mathcal{O}_{Y}$-algebra
isomorphism $\varphi:f^{*}\mathcal{A}\rightarrow\mathcal{O}_{Y}[u]$,
and let 
\[
\Phi=\mathrm{p}_{1}^{*}\varphi\circ\mathrm{p}_{2}^{*}\varphi^{-1}:\mathcal{O}_{Y'}[u]\rightarrow\mathcal{O}_{Y'}[u],\quad u\mapsto au+b
\]
be as in (\ref{eq:Affine-glueing}) above. Since $\theta:W\rightarrow S$
is an $M$-torsor, $\varphi$ also determines an $\mathcal{O}_{Y}$-module
isomorphism $\alpha:f^{*}\mathcal{M}^{\vee}\rightarrow\mathcal{O}_{Y}$
such that $\mathrm{p}_{1}^{*}\alpha\circ\mathrm{p}_{2}^{*}\alpha^{-1}:\mathcal{O}_{Y'}\rightarrow\mathcal{O}_{Y'}$
is the multiplication by $a$. The $\mathbb{G}_{a,S}$-action $\nu$
on $W$ pulls back to a $\mathbb{G}_{a,Y}$-action $\nu\times\mathrm{id}$
on $W\times_{\tilde{S}}Y$. The co-mophism $\eta:\mathcal{O}_{Y}[u]\rightarrow\mathcal{O}_{Y}[u]\otimes\mathcal{O}_{Y}[t]$
of the nontrivial $\mathbb{G}_{a,Y}$-action $\varphi\circ(\nu\times\mathrm{id})\circ(\mathrm{id}\times\varphi^{-1})$
on $\mathrm{Spec}(\mathcal{O}_{Y}[u])$ has the form $u\mapsto u\otimes1+1\otimes\gamma t$
for some nonzero $\gamma\in\Gamma(Y,\mathcal{O}_{Y})$. Letting $\mathcal{I}=\gamma\cdot\mathcal{O}_{Y}$
be the ideal sheaf generated by $\gamma$, $\eta$ factors as 
\[
\eta=(\mathrm{id}\otimes j)\circ\tilde{\eta}:\mathcal{O}_{Y}[u]\rightarrow\mathcal{O}_{Y}[u]\otimes\mathrm{Sym}^{\cdot}\mathcal{I}\rightarrow\mathcal{O}_{Y}[u]\otimes\mathcal{O}_{Y}[t]
\]
where $\tilde{\eta}$ is the co-morphism of an action of the line
bundle $\mathrm{Spec}(\mathrm{Sym}^{\cdot}\mathcal{I})\rightarrow Y$
on $\mathbb{A}_{S}^{1}\times_{S}Y\simeq W\times_{S}Y$ and $j:\mathrm{Sym}^{\cdot}\mathcal{I}\rightarrow\mathcal{O}_{Y}[t]$
is the homomorphism induced by the inclusion $\mathcal{I}\subset\mathcal{O}_{Y}$.
Pulling back to $Y'$, we find that $\mathrm{p}_{2}^{*}\gamma=a\cdot\mathrm{p}_{1}^{*}\gamma$,
which implies that $^{t}\alpha(\gamma)\in\Gamma(Y,f^{*}\mathcal{M})$
is the pull-back $f^{*}s$ to $Y$ of a nonzero global section $s\in\Gamma(S,\mathcal{M})$.
Letting $D=\mathrm{div}_{0}(s)$ be the divisors of zeros of $s$,
we have $\mathcal{M}^{\vee}\simeq\mathcal{O}_{S}(-D)\subset\mathcal{O}_{S}$
and $f^{*}\mathcal{M}^{\vee}\simeq\mathcal{O}_{Y}(-f^{*}D)\subset\mathcal{O}_{Y}$
is equal to the ideal $\mathcal{I}=\gamma\cdot\mathcal{O}_{Y}$. The
global section $f^{*}s$ viewed as a homomorphism $f^{*}\mathcal{M}^{\vee}\rightarrow\mathcal{O}_{Y}$
coincides via these isomorphisms with the inclusion $\gamma\cdot\mathcal{O}_{Y}\hookrightarrow\mathcal{O}_{Y}$.
We can thus rewrite $\eta$ in the form 
\[
\eta=(\mathrm{id}\otimes\mathrm{Sym}^{\cdot}f^{*}s)\circ\tilde{\eta}:\mathcal{O}_{Y}[u]\rightarrow\mathcal{O}_{Y}[u]\otimes\mathrm{Sym}^{\cdot}f^{*}\mathcal{M}^{\vee}\rightarrow\mathcal{O}_{Y}[u]\otimes\mathcal{O}_{Y}[t].
\]
By construction $\tilde{\eta}=(\varphi\otimes\mathrm{id})\circ f^{*}\mu^{\sharp}\circ\varphi^{-1}$
where $f^{*}\mu^{\sharp}$ is the pull-back of the co-morphism $\mu^{\sharp}:\mathcal{A}\rightarrow\mathcal{A}\otimes\mathrm{Sym}^{\cdot}\mathcal{M}^{\vee}$
of the action $\mu:M\times_{S}W\rightarrow W$ of $M$ on $W$. It
follows that the pull-back $f^{*}\nu^{\sharp}$ of the co-morphism
of the action $\nu:\mathbb{G}_{a,S}\times W\rightarrow W$ factors
as 
\[
f^{*}\nu^{\sharp}=(\mathrm{id}\otimes\mathrm{Sym}^{\cdot}f^{*}s)\circ f^{*}\mu^{\sharp}=f^{*}\mathcal{A}\rightarrow f^{*}\mathcal{A}\otimes\mathrm{Sym}^{\cdot}f^{*}\mathcal{M}^{\vee}\rightarrow f^{*}\mathcal{A}\otimes\mathcal{O}_{Y}[t]
\]
This in turn implies that $\nu^{\sharp}$ factors as $(\mathrm{id}\otimes\mathrm{Sym}^{\cdot}s)\circ\mu^{\sharp}:\mathcal{A}\rightarrow\mathcal{A}\otimes\mathrm{Sym}^{\cdot}\mathcal{M}^{\vee}\rightarrow\mathcal{A}\otimes\mathcal{O}_{Y}[t]$
as desired. 
\end{proof}
\begin{rem}
In the setting of Proposition \ref{prop:Torsor-Ga-action}, letting
$U\subset S$ be the complement of the zero locus of $s$, the morphism
$\xi$ restricts to an isomorphism of group schemes $\xi|{}_{U}:\mathbb{G}_{a,U}\rightarrow M|_{U}$
for which $W|_{U}$ equipped with the $\mathbb{G}_{a,U}$-action $\nu|_{U}:\mathbb{G}_{a,U}\times_{U}W|_{U}\rightarrow W|_{U}$
is a $\mathbb{G}_{a,U}$-torsor. This isomorphism class in $H^{1}(U,\mathcal{O}_{U})$
of this $\mathbb{G}_{a,U}$-torsor coincides with the image of the
isomorphism class $g\in H^{1}(S,\mathcal{M})$ of $W$ by the composition
of the restriction homomorphism $\mathrm{res}:H^{1}(S,\mathcal{M})\rightarrow H^{1}(U,\mathcal{M}|_{U})$
with the inverse of the isomorphism $H^{1}(U,\mathcal{O}_{U})\rightarrow H^{1}(U,\mathcal{M}|_{U})$
induced by $s|_{U}$. 
\end{rem}

\subsection{$\mathbb{G}_{a}$-equivariant affine modifications of affine-linear
bundles of rank one}

Recall \cite{Du05} that given an integral scheme $X$ with sheaf
of rational functions $\mathcal{K}_{X}$, an effective Cartier divisor
$D$ on $X$ and a closed subscheme $Z\subset X$ whose ideal sheaf
$\mathcal{I}\subset\mathcal{O}_{X}$ contains $\mathcal{O}_{X}(-D)$,
the \emph{affine modification of} $X$ \emph{with center} $(\mathcal{I},D)$
is the affine $X$-scheme $\sigma:X'=\mathrm{Spec}(\mathcal{O}_{X}[\mathcal{I}/D])\rightarrow X$
where $\mathcal{O}_{X}[\mathcal{I}/D]$ denotes the quotient of the
Rees algebra 
\[
\mathcal{O}_{X}[(\mathcal{I}\otimes\mathcal{O}_{X}(D))]=\bigoplus_{n\geq0}(\mathcal{I}\otimes\mathcal{O}_{X}(D))^{n}t^{n}\subset\mathcal{K}_{X}[t]
\]
of the fractional ideal $\mathcal{I}\otimes\mathcal{O}_{X}(D)\subset\mathcal{K}_{X}$
by the ideal generated by $1-t$. In the case where $X=\mathrm{Spec}(A)$
is affine, $D=\mathrm{div}(f)$ is principal and $Z$ is defined by
an ideal $I\subset A$ containing $f$ then $\tilde{X}$ is isomorphic
to the affine modification $X'=\mathrm{Spec}(A[I/f])$ of $X$ with
center $(I,f)$ in the sense of \cite{KaZa99}. 

Now let $S$ be an integral scheme and let $\theta:W\rightarrow S$
be a locally trivial $\mathbb{A}^{1}$-bundle. Let $C\subset S$ be
an integral Cartier divisor, let $D=\theta^{-1}(C)$ be its inverse
image in $W$ and let $Z\subset D$ be a non empty integral closed
subscheme of $D$ on which $\theta$ restricts to an open embedding
$\theta|_{Z}:Z\hookrightarrow C$. Equivalently, $Z$ is the closure
in $D$ of the image $\alpha(U)$ of a rational section $\alpha:C\rightarrow D$
of the locally trivial $\mathbb{A}^{1}$-bundle $\theta|_{D}:D\rightarrow C$
defined over a non empty open subset $U$ of $C$. The complement
$F$ of $\theta|_{Z}(Z)$ in $C$ is a closed subset of $C$ hence
of $S$. Letting $i:S\setminus F\hookrightarrow S$ be the natural
open embedding, we have the following result:
\begin{lem}
\label{lem:A1bundle-affmod} Let $\sigma:W'\rightarrow W$ be the
affine modification of $W$ with center $(\mathcal{I}_{Z},D)$. Then
the composition $\theta\circ\sigma:W'\rightarrow S$ factors through
a locally trivial $\mathbb{A}^{1}$-bundle $\theta':W'\rightarrow S\setminus F$
in such a way that we have a cartesian diagram \[\label{eq:Modif_diagram} \xymatrix{ W' \ar[d]_{\theta '} \ar[r]^{\sigma} & W \ar[d]^{\theta} \\ S\setminus F \ar[r]^i & S. }\] 
\end{lem}

\begin{proof}
The question being local with respect to a Zariski open cover of $S$
over which $\theta:W\rightarrow S$ becomes trivial, we can assume
without loss of generality that $S=\mathrm{Spec}(A)$, $W=\mathrm{Spec}(A[x])$,
$C=\mathrm{div}(f)$ for some non zero element $f\in A$. The integral
closed subscheme $Z\subset D$ is then defined by an ideal $I$ of
the form $(f,g)$ where $g(x)\in A[x]$ is an element whose image
in $(A/f)[x]$ is a polynomial of degree one in $t$. So $g(x)=a_{0}+a_{1}x+x^{2}fR(x)$
where $a_{0}\in A$, $a_{1}\in A$ has non zero residue class in $A/f$
and $R(x)\in A[x]$. The condition that $\theta|_{Z}:Z\rightarrow C$
is an open embedding implies further that the residue classes $\overline{a}_{0}$
and $\overline{a}_{1}$ of $a_{0}$ and $a_{1}$ in $A/f$ generate
the unit ideal. The complement $F$ of the image of $\theta|_{Z}(Z)$
in $C$ is then equal to the closed subscheme of $C$ with defining
ideal $(\overline{a}_{1})\subset A/f$, hence to the closed subscheme
of $S$ with defining ideal $(f,a_{1})\subset A$. The algebra $A[t][I/f]$
is isomorphic to 
\[
A[x][u]/(g-fu)=A[x][u-x^{2}R(x)]/(a_{0}+a_{1}x-f(u-t^{2}R(x))\simeq A[x][v]/(a_{0}+a_{1}x-fv).
\]
One deduces from this presentation that the morphism $\theta\circ\sigma:W'=\mathrm{Spec}(A[I/f])\rightarrow\mathrm{Spec}(A)$
corresponding to the inclusion $A\rightarrow A[I/f]$ factors through
a locally trivial $\mathbb{A}^{1}$-bundle $\theta':W'\rightarrow S\setminus F$
over the complement of $F$. Namely, since $\overline{a}_{0}$ and
$\overline{a}_{1}$ generate the unit ideal in $A/f$, it follows
that $a_{1}$ and $f$ generate the unit ideal in $A[x][u]/(g-fu)$.
So $W'$ is covered by the two principal affine open subsets 
\begin{align*}
W'_{a_{1}} & \simeq\mathrm{Spec}(A_{a_{1}}[x][v]/(a_{0}+a_{1}x-fv))\simeq\mathrm{Spec}(A_{a_{1}}[v])\simeq S_{a_{1}}\times\mathbb{A}^{1}\\
W'_{f} & \simeq\mathrm{Spec}(A_{f}[x][v]/(a_{0}+a_{1}x-fv))\simeq\mathrm{Spec}(A_{f}[x])\simeq S_{f}\times\mathbb{A}^{1}
\end{align*}
on which $\theta'$ restricts to the projection onto the first factor. 
\end{proof}
With the notation above, $\theta:W\rightarrow S$ and $\theta':W'\rightarrow S\setminus F$
are torsors under the action of line bundles $M=\mathrm{Spec}(\mathrm{Sym}^{\cdot}\mathcal{M}^{\vee})$
and $M'=\mathrm{Spec}(\mathrm{Sym}^{\cdot}{\mathcal{M}'}^{\vee})$
for certain uniquely determined invertible sheaves $\mathcal{M}$
and $\mathcal{M}'$ on $S$ and $S\setminus F$ respectively. 
\begin{lem}
\label{lem:Torsor-affmod} $($\cite[\S 4.3]{DuPhD04}$)$ Let $\sigma:W'\rightarrow W$
be the affine modification of $W$ with center $(\mathcal{I}_{Z},D)$
as is Lemma \ref{lem:A1bundle-affmod}. Then $\mathcal{M}'=\mathcal{M}\otimes_{\mathcal{O}_{S}}\mathcal{O}_{S}(-C)|_{S\setminus F}$
and the cartesian diagram of Lemma \ref{lem:A1bundle-affmod} is equivariant
for the group scheme homomorphism $\xi:M'\rightarrow M$ induced by
the homomorphism $\mathcal{M}\otimes_{\mathcal{O}_{S}}\mathcal{O}_{S}(-C)\rightarrow\mathcal{M}$
obtained by tensoring the inclusion $\mathcal{O}_{S}(-C)\hookrightarrow\mathcal{O}_{S}$
by $\mathcal{M}$. 
\end{lem}

\begin{proof}
Since $M$ and $M'$ are uniquely determined, the question is again
local with respect to a Zariski open cover of $S$ over which $\theta:W\rightarrow S$,
hence $M$, becomes trivial. We can thus assume as in the proof of
Lemma \ref{lem:A1bundle-affmod} that $S=\mathrm{Spec}(A)$, $W=\mathrm{Spec}(A[x])$,
that $C=\mathrm{div}(f)$ for some non zero element $f\in A$ and
that $Z\subset D$ is defined by the ideal $(f,g)$ for some $g=a_{0}+a_{1}x+fx^{2}R(x)\in A[x]$.
Furthermore, the action of $M\simeq\mathbb{G}_{a,S}=\mathrm{Spec}(A[t])$
on $W\simeq S\times\mathbb{A}^{1}$ is the one by translations $x\mapsto x+t$
on the second factor. Let $N=\mathrm{Spec}(\mathrm{Sym}^{\cdot}\mathcal{O}_{S}(C))\simeq\mathrm{Spec}(\mathrm{Sym}^{\cdot}f^{-1}A)$
where $f^{-1}A$ denotes the free sub-$A$-module of the field of
fractions $\mathrm{Frac}(A)$ of $A$ generated by $f^{-1}$. As in
the proof of Proposition \ref{prop:Torsor-Ga-action}, the inclusion
$\mathcal{O}_{S}(-C)=f\cdot\mathcal{O}_{S}\hookrightarrow\mathcal{O}_{S}$
induces a group-scheme homomorphism $\xi:N\rightarrow M$ whose co-morphism
$\xi^{\sharp}$ concides with the inclusion $A[t]\subset\mathrm{Sym}^{\cdot}f^{-1}A=A[(f^{-1}t)]$.
The co-morphism of the corresponding action of $N$ on $W$ is given
by 
\[
A[x]\rightarrow A[x]\otimes A[f^{-1}t],\;x\mapsto x\otimes1+1\otimes t=x\otimes1+f\otimes f^{-1}t.
\]
This action lifts on $W'\simeq\mathrm{Spec}(A[x][v]/(a_{0}+a_{1}x-fv))$
to an action $\nu:N\times_{S}W'\rightarrow W'$ whose co-morphism
\[
A[x][v]/(a_{0}+a_{1}x-fv)\rightarrow A[x][v]/(a_{0}+a_{1}x-fv)\otimes A[f^{-1}t]
\]
is given by $x\mapsto x\otimes1+1\otimes t$ and $v\mapsto v\otimes1+a_{1}\otimes f^{-1}t$.
By construction, the principal open subsets $W'_{a_{1}}\simeq\mathrm{Spec}(A_{a_{1}}[v])\simeq\mathrm{Spec}(A_{a_{1}}[v/a_{1}])$
and $W'_{f}\simeq\mathrm{Spec}(A_{f}[x])\simeq\mathrm{Spec}(A_{f}[x/f])$
of $W'$ equipped with the induced actions of $N|_{S_{a_{1}}}$ and
$N|_{S_{f}}$ respectively are equivariantly isomorphic to $N|_{S_{a_{1}}}$
and $N|_{S_{f}}$ acting on themselves by translations. So $\theta':W'\rightarrow S\setminus F$
is an $N|_{S\setminus F}$-torsor, showing that $\mathcal{M}'=\mathcal{M}\otimes_{\mathcal{O}_{S}}\mathcal{O}_{S}(-C)|_{S\setminus F}$
as desired. 
\end{proof}

\section{Extensions of $\mathbb{G}_{a}$-torsors of Type I: locally trivial
bundles over the blow-up of a point}

Given a surface $S$ and a locally trivial $\mathbb{A}^{1}$-bundle
$\theta:W\rightarrow\tilde{S}$ over the blow-up $\tau:\tilde{S}\rightarrow S$
of a closed point $o$ in the smooth locus of $S$, the restriction
of $W$ over the complement $\tilde{S}\setminus E$ of the exceptional
divisor $E$ of $\tau$ is a locally trivial $\mathbb{A}^{1}$-bundle
$\tau\circ\theta:W\mid_{\tilde{S}\setminus E}\rightarrow\tilde{S}\setminus E\stackrel{\simeq}{\rightarrow}S\setminus\{o\}$.
This observation combined with the following re-interpretation of
an example constructed in \cite{He15} suggests that locally trivial
$\mathbb{A}^{1}$-bundles over the blow-up of closed point $o$ in
the smooth locus of a surface $S$ form a natural class of schemes
in which to search for nontrivial $\mathbb{G}_{a}$-extension of $\mathbb{G}_{a}$-bundles
over punctured surfaces. 
\begin{example}
\label{subsec:Homogeneous-Ga-torsors} Let $o=V(x,y)$ be a global
scheme-theoretic complete intersection closed point in the smooth
locus of a surface $S$. Let $\rho:P\rightarrow S\setminus\{o\}$
and $\pi_{0}:X_{0}\rightarrow S$ be the affine $S$-schemes with
defining sheaves of ideals $(xv-yu-1)$ and $(xr-yq,yp-x(q-1),pr-q(q-1))$
in $\mathcal{O}_{S}[u,v]$ and $\mathcal{O}_{S}[p,q,r]$ respectively.
The morphism of $S$-schemes $j_{0}:P\rightarrow X_{0}$ defined by
$(x,y,u,v)\mapsto(x,y,xu,xv,yv)$ is an open embedding, equivariant
for the $\mathbb{G}_{a,S}$-actions on $P$ and $X_{0}$ associated
with the locally nilpotent $\mathcal{O}_{S}$-derivations $x\partial_{u}+y\partial_{v}$
and $x^{2}\partial_{p}+xy\partial_{q}+y^{2}\partial_{r}$ of $\rho_{*}\mathcal{O}_{P}$
and $(\pi_{0})_{*}\mathcal{O}_{X_{0}}$ respectively. It is straightforward
to check that $\rho:P\rightarrow S\setminus\{o\}$ is a $\mathbb{G}_{a,S_{*}}$-torsor
and that $\pi_{0}:X_{0}\rightarrow S$ is a $\mathbb{G}_{a}$-extension
of $P$ whose fiber over $o$ is isomorphic to the smooth affine quadric
$\{pr-q(q-1)=0\}\subset\mathbb{A}_{\kappa}^{3}$. Viewing the blow-up
$\tilde{S}$ of $o$ as the closed subscheme of $S\times_{k}\mathrm{Proj}(k[u_{0},u_{1}])$
with equation $xu_{1}-yu_{0}=0$, the morphism of $S$-schemes $\theta:X_{0}\rightarrow\tilde{S}$
defined by 
\[
(x,y,p,q,r)\mapsto((x,y),[x:y])=((x,y),[q:r])=((x,y),[p:q-1])
\]
is a locally trivial $\mathbb{A}^{1}$-bundle, actually a torsor under
the line bundle corresponding to the invertible sheaf $\mathcal{O}_{\tilde{S}}(-2E)$,
where $E\simeq\mathbb{P}_{\kappa}^{1}$ denotes the exceptional divisor
of the blow-up. 
\end{example}

\begin{notation}
Given a surface $S$ and a closed point $o$ in the smooth locus of
$S$, with residue field $\kappa$, we denote by $\tau:\tilde{S}\rightarrow S$
be the blow-up of $o$, with exceptional divisor $E\simeq\mathbb{P}_{\kappa}^{1}$.
We identify $\tilde{S}\setminus E$ and $S_{*}=S\setminus\{o\}$ by
the isomorphism induced by $\tau$. For every $\ell\in\mathbb{Z}$,
we denote by $M(\ell)=\mathrm{Spec}(\mathrm{Sym}^{\cdot}\mathcal{O}_{\tilde{S}}(-\ell E))$
the line bundle on $\tilde{S}$ corresponding to the invertible sheaf
$\mathcal{O}_{\tilde{S}}(\ell E)$. 
\end{notation}

The aim of this section is to give a classification of all possible
$\mathbb{G}_{a}$-equivariant extensions of Type I of a given $\mathbb{G}_{a}$-torsor
$\rho:P\rightarrow S_{*}$, that is $\mathbb{G}_{a}$-extensions $\pi:W\rightarrow S$
that factor through locally trivial $\mathbb{A}^{1}$-bundles $\theta:W\rightarrow\tilde{S}$. 

\subsection{Existence of $\mathbb{G}_{a}$-extensions of Type I}

By virtue of Propositions \ref{prop:A1bundle-torsor} and \ref{prop:Torsor-Ga-action},
there exists a one-to-one correspondence between $\mathbb{G}_{a}$-equivariant
extensions of a $\mathbb{G}_{a}$-torsor $\rho:P\rightarrow S_{*}$
that factor through a locally trivial $\mathbb{A}^{1}$-bundle $\theta:W\rightarrow\tilde{S}$
and pairs $(M,\xi)$ consisting of an $M$-torsor $\theta:W\rightarrow\tilde{S}$
for some line bundle $M$ on $\tilde{S}$ and a group scheme homomorphism
$\xi:\mathbb{G}_{a,\tilde{S}}\rightarrow M$ restricting to an isomorphism
over $\tilde{S}\setminus E$, such that $W$ equipped with the $\mathbb{G}_{a,\tilde{S}}$-action
deduced by composition with $\xi$ restricts on $S_{*}=\tilde{S}\setminus E$
to a $\mathbb{G}_{a,S_{*}}$-torsor $\theta\mid_{S_{*}}:W\mid_{S_{*}}\rightarrow S_{*}$
isomorphic to $\rho:P\rightarrow S_{*}$. The condition that $\xi:\mathbb{G}_{a,\tilde{S}}\rightarrow M$
restricts to an isomorphism outside $E$ implies that $M\simeq M(\ell)$
for some $\ell$, which is necessarily non negative, and that $\xi$
is induced by the canonical global section of $\mathcal{O}_{\tilde{S}}(\ell E)$
with divisor $\ell E$. 
\begin{prop}
\label{prop:Blowup-Extension} Let $\rho:P\rightarrow S_{*}$ be a
$\mathbb{G}_{a,S_{*}}$-torsor. Then there exists an integer $\ell_{0}\geq0$
depending on $P$ only such that for every $\ell\geq\ell_{0}$, $P$
admits a $\mathbb{G}_{a}$-extension to a uniquely determined $M(\ell)$-torsor
$\theta_{\ell}:W(P,\ell)\rightarrow\tilde{S}$ equipped with the $\mathbb{G}_{a,\tilde{S}}$-action
induced by the canonical global section $s_{\ell}\in\Gamma(\tilde{S},\mathcal{O}_{\tilde{S}}(\ell E))$
with divisor $\ell E.$ 
\end{prop}

\begin{proof}
The invertible sheaves $\mathcal{O}_{\tilde{S}}(nE)$, $n\geq0$,
form an inductive system of sub-$\mathcal{O}_{\tilde{S}}$-modules
of the sheaf $\mathcal{K}_{\tilde{S}}$ of rational function on $\tilde{S}$,
where for each $n$, the injective transition homomorphism $j_{n,n+1}:\mathcal{O}_{\tilde{S}}(nE)\hookrightarrow\mathcal{O}_{\tilde{S}}((n+1)E)$
is obtained by tensoring the canonical section $\mathcal{O}_{\tilde{S}}\rightarrow\mathcal{O}_{\tilde{S}}(E)$
with divisor $E$ with $\mathcal{O}_{\tilde{S}}(nE)$. Let $i:S_{*}=\tilde{S}\setminus E\hookrightarrow\tilde{S}$
be the open inclusion. Since $E$ is a Cartier divisor, it follows
from \cite[Théorème 9.3.1]{EGA1} that $i_{*}\mathcal{O}_{S_{*}}\simeq\mathrm{colim}_{n\geq0}\mathcal{O}_{\tilde{S}}(nE)$.
Furthermore, since $E\simeq\mathbb{P}_{\kappa}^{1}$ is the exceptional
divisor of $\tau:\tilde{S}\rightarrow S$, we have $\mathcal{O}_{\tilde{S}}(E)|_{E}\simeq\mathcal{O}_{\mathbb{P}_{\kappa}^{1}}(-1)$,
and the long exact sequence of cohomology for the short exact sequence
\begin{equation}
0\rightarrow\mathcal{O}_{\tilde{S}}(nE)\rightarrow\mathcal{O}_{\tilde{S}}((n+1)E)\rightarrow\mathcal{O}_{\tilde{S}}((n+1)E)|_{E}\rightarrow0,\quad n\geq0,\label{eq:fond-sequence}
\end{equation}
combined with the vanishing of $H^{0}(\mathbb{P_{\kappa}}^{1},\mathcal{O}_{\mathbb{P}_{\kappa}^{1}}(-n-1))$
for every $n\geq0$ implies that the transition homomorphisms 
\[
H^{1}(j_{n,n+1}):H^{1}(\tilde{S},\mathcal{O}_{\tilde{S}}(nE))\rightarrow H^{1}(\tilde{S},\mathcal{O}_{\tilde{S}}((n+1)E)),\;n\geq0,
\]
are all injective. By assumption, $S$ whence $\tilde{S}$ is noetherian,
and $i:S_{*}\rightarrow\tilde{S}$ is an affine morphism as $E$ is
a Cartier divisor on $\tilde{S}$. We thus deduce from \cite[Theorem 8]{Ke80}
and \cite[Corollaire 1.3.3]{EGA3} that the canonical homomorphism
\begin{equation}
\psi:\mathrm{colim}_{n\geq0}H^{1}(\tilde{S},\mathcal{O}_{\tilde{S}}(nE))\rightarrow H^{1}(S_{*},\mathcal{O}_{S_{*}})\label{eq:cano-iso}
\end{equation}
obtained as the composition of the canonical homomorphisms 
\[
\mathrm{colim}_{n\geq0}H^{1}(\tilde{S},\mathcal{O}_{\tilde{S}}(nE))\rightarrow H^{1}(\tilde{S},\mathrm{colim}_{n\geq0}\mathcal{O}_{\tilde{S}}(nE))=H^{1}(\tilde{S},i_{*}\mathcal{O}_{S_{*}})
\]
and $H^{1}(\tilde{S},i_{*}\mathcal{O}_{S_{*}})\rightarrow H^{1}(S_{*},\mathcal{O}_{S_{*}})$
is an isomorphism. 

Let $g\in H^{1}(S_{*},\mathcal{O}_{S_{*}})$ be the isomorphism class
of the $\mathbb{G}_{a,S_{*}}$-torsor $\rho:P\rightarrow S_{*}$.
If $g=0$, then since $\psi$ is an isomorphism, we have $\psi^{-1}(g)=0$
and, since the homomorphisms $H^{1}(j_{n,n+1})$ are injective, it
follows that $\psi^{-1}(g)$ is represented by the zero sequence $(0)_{n}\in H^{1}(\tilde{S},\mathcal{O}_{\tilde{S}}(nE))$,
$n\geq0$ . Consequently, the only $\mathbb{G}_{a}$-extensions of
$P$ are the line bundles $W(P,\ell)=M(\ell)$, $\ell\geq0$, each
equipped with the $\mathbb{G}_{a,\tilde{S}}$-action induced by its
canonical global section $s_{\ell}\in\Gamma(\tilde{S},\mathcal{O}_{\tilde{S}}(\ell E))$. 

Otherwise, if $g\neq0$, then $h=\psi^{-1}(g)\neq0$, and since the
homomorphisms $H^{1}(j_{n,n+1})$, $n\geq0$ are injective, it follows
that there exists a unique minimal integer $\ell_{0}$ such that $h$
is represented by the sequence 
\begin{equation}
h_{n}=H^{1}(j_{n-1,n})\circ\cdots\circ H^{1}(j_{\ell_{0},\ell_{0}+1})(h_{\ell_{0}})\in H^{1}(\tilde{S},\mathcal{O}_{\tilde{S}}(nE)),\;n\geq\ell_{0}\label{eq:colim-sequence}
\end{equation}
for some non zero $h_{\ell_{0}}\in H^{1}(\tilde{S},\mathcal{O}_{\tilde{S}}(\ell_{0}E))$.
It then follows from Proposition \ref{prop:Torsor-Ga-action} that
for every $\ell\geq\ell_{0}$, the $M(\ell)$-torsor $\theta_{\ell}:W(P,\ell)\rightarrow\tilde{S}$
with isomorphism class $h_{\ell}$ equipped with the $\mathbb{G}_{a,\tilde{S}}$-action
induced by the canonical global section $s_{\ell}\in\Gamma(\tilde{S},\mathcal{O}_{\tilde{S}}(\ell E))$
is a $\mathbb{G}_{a}$-extension of $P$. 

Conversely, for every $\mathbb{G}_{a}$-extension of $P$ into an
$M(\ell)$-torsor $\theta:W\rightarrow\tilde{S}$ equipped with the
$\mathbb{G}_{a,\tilde{S}}$-action induced by the canonical global
section $s_{\ell}\in\Gamma(\tilde{S},\mathcal{O}_{\tilde{S}}(\ell E))$,
it follows from Proposition \ref{prop:Torsor-Ga-action} again that
the image of the isomorphism class $h_{\ell}\in H^{1}(\tilde{S},\mathcal{O}_{\tilde{S}}(\ell E))$
of $W$ in $H^{1}(\tilde{S}\setminus E,\mathcal{O}_{\tilde{S}}(\ell E)|_{\tilde{S}\setminus E})\simeq H^{1}(S_{*},\mathcal{O}_{S_{*}})$
is equal to $g$. Letting $h\in\mathrm{colim}_{n\geq0}H^{1}(\tilde{S},\mathcal{O}_{\tilde{S}}(nE))$
be the element represented by the sequence 
\[
h_{n}=(H^{1}(j_{n-1,n}\circ\cdots\circ j_{\ell,\ell+1})(h_{\ell}))_{n\geq\ell}\in H^{1}(\tilde{S},\mathcal{O}_{\tilde{S}}(nE)),\;n\geq\ell
\]
we have $\psi(h)=g$ and since $\psi$ is an isomorphism, we conclude
that $W\simeq W(P,\ell)$. 
\end{proof}

\subsection{\label{subsec:Ga-blowup-affine}$\mathbb{G}_{a}$-extensions with
affine total spaces }

The extensions $\theta:W\rightarrow\tilde{S}$ we get from Proposition
\ref{prop:Blowup-Extension} are not necessarily affine over $S$.
In this subsection we establish a criterion for affineness which we
then use to characterize all extensions $\theta:W\rightarrow\tilde{S}$
of a $\mathbb{G}_{a}$-torsor $\rho:P\rightarrow S_{*}$ whose total
spaces $W$ are affine over $S$.
\begin{lem}
\label{lem:Affineness-Criterion} Let $S=\mathrm{Spec}(A)$ be an
affine surface and let $o=V(x,y)$ be a global scheme-theoretic complete
intersection point in the smooth locus of $S$ . Let $\tau:\tilde{S}\rightarrow S$
be the blow-up of $o$ with exceptional divisor $E$ and let $\theta:W\rightarrow\tilde{S}$
be an $M(\ell)$-torsor for some $\ell\geq0$. Then the following
hold:

a) $H^{1}(W,\mathcal{O}_{W})=0$. 

b) If $H^{1}(W,\theta^{*}\mathcal{O}_{\tilde{S}}(\ell E))=0$ for
some $\ell\geq2$ then $W$ is an affine scheme. 
\end{lem}

\begin{proof}
Since $o$ is a scheme-theoretic complete intersection, we can identify
$\tilde{S}$ with the closed subvariety of $S\times_{k}\mathbb{P}_{k}^{1}=S\times_{k}\mathrm{Proj}(k[t_{0},t_{1}])$
defined by the equation $xt_{1}-yt_{0}=0$. The restriction $p:\tilde{S}\rightarrow\mathbb{P}_{k}^{1}$
of the projection to the second factor is an affine morphism. More
precisely, letting $U_{0}=\mathbb{P}_{k}^{1}\setminus\{[1:0]\}\simeq\mathrm{Spec}(k[z])$
and $U_{\infty}=\mathbb{P}_{k}^{1}\setminus\{[0:1]\}\simeq\mathrm{Spec}(k[z'])$
be the standard affine open cover of $\mathbb{P}_{k}^{1}$, we have
$p^{-1}(U_{0})\simeq\mathrm{Spec}(A[z]/(x-yz)$ and $p^{-1}(U_{\infty})\simeq\mathrm{Spec}(A[z']/(y-xz'))$.
The exceptional divisor $E\simeq\mathbb{P}_{\kappa}^{1}$ of $\tau:\tilde{S}\rightarrow S$
is a flat quasi-section of $p$ with local equations $y=0$ and $x=0$
in the affine charts $p^{-1}(U_{0})$ and $p^{-1}(U_{\infty})$ respectively.
Every $M(\ell)$-torsor $\theta:W\rightarrow\tilde{S}$ for some $\ell\geq0$
is isomorphic to the scheme obtained by gluing $W_{0}=p^{-1}(U_{0})\times\mathrm{Spec}(k[u])$
with $W_{\infty}=p^{-1}(U_{\infty})\times\mathrm{Spec}(k[u'])$ over
$U_{0}\cap U_{\infty}$ by an isomorphism induced by a $k$-algebra
isomorphism of the form 
\[
A[(z')^{\pm1}]/(y-xz')[u']\ni(z',u')\mapsto(z^{-1},z^{\ell}u+p)\in A[z^{\pm1}]/(x-yz)[u]
\]
for some $p\in A[z^{\pm1}]/(x-yz)$. Since $H^{1}(W,\mathcal{O}_{W})\simeq\check{H}^{1}(W,\mathcal{O}_{W})\simeq\check{H}^{1}(\{W_{0},W_{\infty}\},\mathcal{O}_{W})$,
it is enough in order to prove a) to check that every \v{C}ech $1$-cocycle
$g$ with value in $\mathcal{O}_{W}$ for the covering of $W$ by
the affine open subsets $W_{0}$ and $W_{\infty}$ is a coboundary.
Viewing $g$ as an element $g=g(z^{\pm1},u)\in A[z^{\pm1}]/(x-yz)[u]$,
it is enough to show that every monomial $g_{s}=hz^{r}u^{s}$ where
$h\in A$, $r\in\mathbb{Z}$ and $s\in\mathbb{Z}_{\geq0}$ is a coboundary,
which is the case if and only if there exist $a(z,u)\in A[z]/(f-gz)[u]$
and $b(z',u')\in A[z']/(y-xz')[u']$ such that $g=b(z^{-1},z^{\ell}u+p)-b(z,u)$.
If $r\geq0$ then $g\in A[z]/(x-yz)[u]$ is a coboundary. We thus
assume from now on that $r<0$. Suppose that $s>0$. Then we can write
$u^{s}=z^{-\ell s}(z^{\ell}u+p)^{s}-R(u)$ where $R\in A[z^{\pm1}]/(x-yz)[u]$
is polynomial whose degree in $u$ is strictly less than $s$. Then
since $r<0$, 
\begin{align*}
hz^{r}u^{s} & =hz^{r-\ell s}(z^{\ell}u+p)^{s}-hz^{r}R(u)\\
 & =b(z^{-1},z^{\ell}u+p)-hz^{r}R(u)
\end{align*}
where $b(z',u')=h(z')^{-r+\ell s}(u')^{s}\in A[z']/(y-xz')[u']$.
So $g_{s}$ is a coboundary if and only if $-hz^{r}R(u)$ is. By induction,
we only need to check that every monomial $g_{0}=hz^{r}\in A[z^{\pm1}]/(x-yz)[u]$
of degree $0$ in $u$ is a coboundary. But such a cocycle is simply
the pull-back to $W$ of a \v{C}ech $1$-cocycle $h_{0}$ with value
in $\mathcal{O}_{\tilde{S}}$ for the covering of $\tilde{S}$ by
the affine open subsets $p^{-1}(U_{0})$ and $p^{-1}(U_{\infty})$.
Since the canonical homomorphism 
\[
H^{1}(S,\mathcal{O}_{S})=H^{1}(S,\tau_{*}\mathcal{O}_{\tilde{S}})\rightarrow H^{1}(\tilde{S},\mathcal{O}_{\tilde{S}})\simeq\check{H}^{1}(\{p^{-1}(U_{0}),p^{-1}(U_{\infty})\},\mathcal{O}_{\tilde{S}})
\]
is an isomorphism and $H^{1}(S,\mathcal{O}_{S})=0$ as $S$ is affine,
we conclude that $h_{0}$ is a coboundary, hence that $g_{0}$ is
a coboundary too. This proves a). 

Now suppose that $H^{1}(W,\theta^{*}\mathcal{O}_{\tilde{S}}(\ell E))=0$
for some $\ell\geq2$. Let $\eta:V\rightarrow\mathbb{P}_{k}^{1}$
be a non trivial $\mathcal{O}_{\mathbb{P}_{k}^{1}}(-\ell)$-torsor
and consider the fiber product $W\times_{p\circ\theta,\mathbb{P}_{k}^{1},\eta}V$:
\[\xymatrix@C-3ex{ & W\times_{p\circ\theta,\mathbb{P}^1_k,\eta } V \ar[dl] \ar[dr] \\ W \ar[dr]_{p\circ \theta} & & V \ar[dl]^{\eta} \\ & \mathbb{P}^1_k}\]
By virtue of \cite[Proposition 3.1]{Du15}, $V$ is an affine surface.
Since $p\circ\theta:W\rightarrow\mathbb{P}_{k}^{1}$ is an affine
morphism, so is $\mathrm{pr}_{V}:W\times_{\mathbb{P}_{k}^{1}}V\rightarrow V$
and hence, $W\times_{\mathbb{P}_{k}^{1}}V$ is an affine scheme. On
the other hand, since $p^{*}\mathcal{O}_{\mathbb{P}_{k}^{1}}(-1)\simeq\mathcal{O}_{\tilde{S}}(E)$,
the projection $\mathrm{pr}_{W}:W\times_{\mathbb{P}_{k}^{1}}V\rightarrow W$
is a $\theta^{*}M(\ell)$-torsor, hence is isomorphic to the trivial
one $q:\theta^{*}M(\ell)\rightarrow W$ by hypothesis. So $W$ is
isomorphic to the zero section of $\theta^{*}M(\ell)$, which is a
closed subscheme of the affine scheme $W\times_{\mathbb{P}_{k}^{1}}V$,
hence an affine scheme. 
\end{proof}
We are now ready to prove the following characterization: 
\begin{thm}
\label{prop:Main-Ext-Prop} A $\mathbb{G}_{a,S_{*}}$-torsor $\rho:P\rightarrow S_{*}$
admits a $\mathbb{G}_{a}$-extension to a locally trivial $\mathbb{A}^{1}$-bundle
whose total space is affine over $S$ if and only if for every Zariski
open neighborhood $U$ of $o$, $P\times_{S_{*}}U\rightarrow U_{*}=U\setminus\{o\}$
is a non trivial $\mathbb{G}_{a,U_{*}}$-torsor. 

When it exists, the corresponding locally trivial $\mathbb{A}^{1}$-bundle
$\theta:W\rightarrow\tilde{S}$ is unique and is an $M(\ell_{0})$-torsor
for some $\ell_{0}\geq2$, whose restriction to $E\simeq\mathbb{P}_{\kappa}^{1}$
is a non trivial $\mathcal{O}_{\mathbb{P}_{\kappa}^{1}}(-\ell_{0})$-torsor. 
\end{thm}

\begin{proof}
The scheme $W$ is affine over $S$ if and only if its restriction
$W|_{E}$ over $E\subset\tilde{S}$ is a nontrivial torsor. Indeed,
if $W|_{E}$ is a trivial torsor then it is a line bundle over $E\simeq\mathbb{P}_{\kappa}^{1}$.
Its zero section is then a proper curve contained in the fiber of
$\pi=\tau\circ\theta:W\rightarrow S$, which prevents $\pi$ from
being an affine morphism. Conversely, if $W|_{E}$ is nontrivial,
then it is a torsor under a uniquely determined line bundle $\mathcal{O}_{\mathbb{P}_{\kappa}^{1}}(-m)$
for some $m\geq2$ necessarily. Since by construction $\pi$ restricts
over $S_{*}$ to $\rho:P\rightarrow S_{*}$ which is an affine morphism,
$\pi$ is affine if and only if there exists an open neighborhood
$U$ of $o$ in $S$ such that $\pi^{-1}(U)$ is affine. Replacing
$S$ by a suitable affine open neighborhood of $o$, we can therefore
assume without loss of generality that $S=\mathrm{Spec}(A)$ is affine
and that $o$ is a scheme-theoretic complete intersection $o=V(x,y)$
for some elements $x,y\in A$. By virtue of \cite[Proposition 3.1]{Du15}
every nontrivial $\mathcal{O}_{\mathbb{P}_{\kappa}^{1}}(-m)$-torsor,
$m\geq2$ , has affine total space. The Cartier divisor $D=W|_{E}$
in $W$ is thus an affine surface, and so $H^{1}(D,\mathcal{O}_{W}((n+1)D)|_{D})=0$
for every $n\in\mathbb{Z}$. By a) in Lemma \ref{lem:Affineness-Criterion},
$H^{1}(W,\mathcal{O}_{W})=0$, and we deduce successively from the
long exact sequence of cohomology for the short exact sequence
\[
0\rightarrow\mathcal{O}_{W}(nD)\rightarrow\mathcal{O}_{W}((n+1)D)\rightarrow\mathcal{O}_{W}((n+1)D)\mid_{D}\rightarrow0
\]
in the case $n=0$ and then $n=1$ that $H^{1}(W,\mathcal{O}_{W}(D))=H^{1}(W,\mathcal{O}_{W}(2D))=0$.
Since $\mathcal{O}_{W}(2D)\simeq\theta^{*}\mathcal{O}_{\tilde{S}}(2E)$,
we conclude from b) in the same lemma that $W$ is affine. 

The condition that $P\times_{S_{*}}U\rightarrow U_{*}$ is nontrivial
for every open neighborhood $U$ of $o$ is necessary for the existence
of an extension $\theta:W\rightarrow\tilde{S}$ of $P$ for which
$W|_{E}$ is a nontrivial torsor. Indeed, if there exists a Zariski
open neighborhood $U$ of $o$ such that the restriction of $P$ over
$U_{*}$ is the trivial $\mathbb{G}_{a,U_{*}}$-torsor, then the image
in $H^{1}(U_{*},\mathcal{O}_{U_{*}})$ of the isomorphism class $g$
of $P$ is zero and so, arguing as in the proof of Proposition \ref{prop:Blowup-Extension},
every $\mathbb{G}_{a}$-extension $\theta:W\rightarrow\tilde{S}$
restricts on $\tau^{-1}(U)$ to the trivial $M(\ell)|_{\tau^{-1}(U)}$-torsor
$M(\ell)|_{\tau^{-1}(U)}\rightarrow\tau^{-1}(U)$, hence to a trivial
torsor on $E\subset\tau^{-1}(U)$. 

Now suppose that $\rho:P\rightarrow S_{*}$ is a $\mathbb{G}_{a,S_{*}}$-torsor
with isomorphism class $g\in H^{1}(S_{*},\mathcal{O}_{S_{*}})$ such
that $P\times_{S_{*}}U\rightarrow U_{*}$ is non trivial for every
open neighborhood $U$ of $o$. The inverse image $h=\psi^{-1}(g)\in\mathrm{colim}_{n\geq0}H^{1}(\tilde{S},\mathcal{O}_{\tilde{S}}(nE))$
of $g$ by the isomorphism (\ref{eq:cano-iso}) is represented by
a sequence of nonzero elements $h_{n}\in H^{1}(\tilde{S},\mathcal{O}_{\tilde{S}}(nE))$
as in (\ref{eq:colim-sequence}) above. By the long exact sequence
of cohomology of the short exact sequence (\ref{eq:fond-sequence}),
the image $\overline{h}_{n}$ of $h_{n}$ in $H^{1}(E,\mathcal{O}_{\tilde{S}}(nE)|_{E})\simeq H^{1}(\mathbb{P}^{1},\mathcal{O}_{\mathbb{P}_{\kappa}^{1}}(-n))$
is nonzero if and only if $h_{n}$ is not in the image of the injective
homomorphism $H^{1}(j_{n,n-1})$. Since $\overline{h}_{n}$ coincides
with the isomorphism class of the restriction $W_{n}|_{E}$ of an
$M(n)$-torsor $\theta_{n}:W_{n}\rightarrow\tilde{S}$ with isomorphism
class $h_{n}$, we conclude that there exists a unique $\ell_{0}\geq2$
such that the restriction to $E$ of an $M(\ell_{0})$-torsor $\theta_{\ell_{0}}:W_{\ell_{0}}\rightarrow\tilde{S}$
with isomorphism class $h_{\ell_{0}}\in H^{1}(\tilde{S},\mathcal{O}_{\tilde{S}}(\ell_{0}E))$
is a nontrivial $\mathcal{O}_{\mathbb{P}_{\kappa}^{1}}(-\ell_{0})$-torsor. 
\end{proof}

\subsection{Examples}

In this subsection, we consider $\mathbb{G}_{a}$-torsors of the punctured
affine plane. So $S=\mathbb{A}^{2}=\mathrm{Spec}(k[x,y])$, $o=(0,0)$
and $\mathbb{A}_{*}^{2}=\mathbb{A}^{2}\setminus\{o\}$. We let $\tau:\tilde{\mathbb{A}}^{2}\rightarrow\mathbb{A}^{2}$
be the blow-up of $o$, with exceptional divisor $E\simeq\mathbb{P}^{1}$
and we let $i:\mathbb{A}_{*}^{2}\hookrightarrow\tilde{\mathbb{A}}^{2}$
be the immersion of $\mathbb{A}_{*}^{2}$ as the open subset $\tilde{\mathbb{A}}^{2}\setminus E$.
We further identify $\tilde{\mathbb{A}}^{2}$ with the total space
$f:\tilde{\mathbb{A}}^{2}\rightarrow\mathbb{P}^{1}$ of the line bundle
$\mathcal{O}_{\mathbb{P}^{1}}(-1)$ in such a way that $E$ corresponds
to the zero section of this line bundle.

\subsubsection{A simple case: homogeneous $\mathbb{G}_{a}$-torsors }

Following \cite[\S 1.3]{DF14}, we say that a non trivial $\mathbb{G}_{a,\mathbb{A}_{*}^{2}}$-torsor
$\rho:P\rightarrow\mathbb{A}_{*}^{2}$ is homogeneous if it admits
a lift of the $\mathbb{G}_{m}$-action $\lambda\cdot\left(x,y\right)=\left(\lambda x,\lambda y\right)$
on $\mathbb{A}_{*}^{2}$ which is locally linear on the fibers of
$\rho$. By \cite[Proposition 1.6]{DF14}, this is the case if and
only if the isomorphism class $g$ of $P$ in $H^{1}(\mathbb{A}_{*}^{2},\mathcal{O}_{\mathbb{A}_{*}^{2}})$
can be represented on the open covering of $\mathbb{A}_{*}^{2}$ by
the principal open subsets $\mathbb{A}_{x}^{2}$ and $\mathbb{A}_{y}^{2}$
by a \v{C}ech $1$-cocycle of the form $x^{-m}y^{-n}p\left(x,y\right)$
where $m,n\geq0$ and $p(x,y)\in k[x,y]$ is a homogeneous polynomial
of degree $r\leq m+n-2$. Equivalently, $P$ is isomorphic the $\mathbb{G}_{a,\mathbb{A}_{*}^{2}}$-torsor
\[
\rho=\mathrm{pr}_{x,y}:P_{m,n,p}=\left\{ x^{m}v-y^{n}u=p(x,y)\right\} \setminus\{x=y=0\}\rightarrow\mathbb{A}_{*}^{2},
\]
which admits an obvious lift $\lambda\cdot\left(x,y,u,v\right)=\left(\lambda x,\lambda y,\lambda^{m-d}u,\lambda^{n-d}v\right)$,
where $d=m+n-r$, of the $\mathbb{G}_{m}$-action on $\mathbb{A}_{*}^{2}$.
Let $q:\mathbb{A}_{*}^{2}\rightarrow\mathbb{A}_{*}^{2}/\mathbb{G}_{m}=\mathbb{P}^{1}$
be the quotient morphism of the aforementioned $\mathbb{G}_{m}$-action
on $\mathbb{A}_{*}^{2}$. Then it follows from \cite[Example 1.8]{DF14}
that the inverse image by the canonical isomorphism 
\[
\bigoplus_{k\in\mathbb{Z}}H^{1}(\mathbb{P}^{1},\mathcal{O}_{\mathbb{P}}(k))\simeq H^{1}(\mathbb{P}^{1},q_{*}\mathcal{O}_{\mathbb{A}_{*}^{2}})\rightarrow H^{1}(\mathbb{A}_{*}^{2},\mathcal{O}_{\mathbb{A}_{*}^{2}})
\]
of the isomorphism class $g$ of such an homogeneous torsor is an
element $h$ of $H^{1}(\mathbb{P}^{1},\mathcal{O}_{\mathbb{P}}(-d))$.
Furthermore, the $\mathbb{G}_{m}$-equivariant morphism $\rho:P\rightarrow\mathbb{A}_{*}^{2}$
descends to a locally trivial $\mathbb{A}^{1}$-bundle $\overline{\rho}:P/\mathbb{G}_{m}\rightarrow\mathbb{P}^{1}=\mathbb{A}_{*}^{2}/\mathbb{G}_{m}$
which is an $\mathcal{O}_{\mathbb{P}^{1}}(-d)$-torsor with isomorphism
class $h\in H^{1}(\mathbb{P}^{1},\mathcal{O}_{\mathbb{P}}(-d))$.

Since $f^{*}\mathcal{O}_{\mathbb{P}^{1}}(-d)\simeq\mathcal{O}_{\tilde{\mathbb{A}}^{2}}(dE)$,
the fiber product $W(P,d)=\tilde{\mathbb{A}}^{2}\times_{\mathbb{P}^{1}}P/\mathbb{G}_{m}$
is equipped via the restriction of the first projection with the structure
of an $M(d)$-torsor $\theta:W(P,d)\rightarrow\tilde{\mathbb{A}}^{2}$
with isomorphism class $f^{*}h\in H^{1}(\tilde{\mathbb{A}}^{2},\mathcal{O}_{\tilde{\mathbb{A}}^{2}}(dE))$.
On the other other hand, $W(P,d)$ is a line bundle over $P/\mathbb{G}_{m}$
via the second projection, hence is an affine threefold as $P/\mathbb{G}_{m}$
is affine. By construction, we have a commutative diagram \[\xymatrix@C+2ex@R-1ex{ & W(P,d) \ar@{->}'[d]^-{\theta}[dd] \ar[dr] \\ P \ar[dd]_{\rho} \ar[ur]^{j} \ar[rr] & & P/\mathbb{G}_m \ar[dd]_{\overline{\rho}} \\ & \tilde{\mathbb{A}}^2 \ar[dr]^{f} \\ \mathbb{A}^2_{*} \ar[ur]^{i} \ar[rr]^{q} && \mathbb{P}^1}\]in
which each square is cartesian. In other words, $W(P,d)$ is obtained
from the $\mathbb{G}_{m}$-torsor $P\rightarrow P/\mathbb{G}_{m}$
by ``adding the zero section''. The open embedding $j:P\hookrightarrow W(P,d)$
is equivariant for the $\mathbb{G}_{a}$-action on $W(P,d)$ induced
by the canonical global section of $\mathcal{O}_{\tilde{\mathbb{A}}^{2}}(dE)$
with divisor $dE$ (see Proposition \ref{prop:Torsor-Ga-action}).
By Theorem \ref{prop:Main-Ext-Prop} , $\theta:W(P,d)\rightarrow\tilde{\mathbb{A}}^{2}$
is the unique $\mathbb{G}_{a}$-extension of $\rho:P\rightarrow\mathbb{A}_{*}^{2}$
with affine total space. 

In the simplest case $d=2$, the unique homogeneous $\mathbb{G}_{a,\mathbb{A}_{*}^{2}}$-torsor
is the geometric quotient $\mathrm{SL}_{2}\rightarrow\mathrm{SL}_{2}/\mathbb{G}_{a}$
of the group $\mathrm{SL}_{2}$ by the action of its subgroup of upper
triangular unipotent matrices equipped with the diagonal $\mathbb{G}_{m}$-action,
and we recover Example \ref{subsec:Homogeneous-Ga-torsors}. 

\subsubsection{General case}

Here, given an arbitrary non trivial $\mathbb{G}_{a}$-torsor $\rho:P\rightarrow\mathbb{A}_{*}^{2}$,
we describe a procedure to explicitly determine the unique $\mathbb{G}_{a}$-extension
$\theta:W\rightarrow\tilde{\mathbb{A}}^{2}$ of $P$ with affine total
space $W$ from a \v{C}ech $1$-cocycle $x^{-m}y^{-n}p\left(x,y\right)$,
where $m,n\geq0$ and $p(x,y)\in k[x,y]$ is a non zero polynomial
of degree $r\leq m+n-2$, representing the isomorphism class $g\in H^{1}(\mathbb{A}_{*}^{2},\mathcal{O}_{\mathbb{A}_{*}^{2}})$
of $P$ on the open covering of $\mathbb{A}_{*}^{2}$ by the principal
open subsets $\mathbb{A}_{x}^{2}$ and $\mathbb{A}_{y}^{2}$. 

Write $p(x,y)=p_{d}+p_{d+1}+\cdots+p_{r}$ where the $p_{i}\in k[x,y]$
are the homogeneous components of $p$, and $p_{d}\neq0$. In the
decomposition 
\[
H^{1}(\mathbb{A}_{*}^{2},\mathcal{O}_{\mathbb{A}_{*}^{2}})\simeq H^{1}(\mathbb{P}^{1},q_{*}\mathcal{O}_{\mathbb{A}_{*}^{2}})\simeq\bigoplus_{s\in\mathbb{Z}}H^{1}(\mathbb{P}^{1},\mathcal{O}_{\mathbb{P}^{1}}(s))
\]
a non zero homogeneous component $x^{-m}y^{-n}p_{i}$ of $x^{-m}y^{-n}p(x,y)$
corresponds to a non zero element of $H^{1}(\mathbb{P}^{1},\mathcal{O}_{\mathbb{P}^{1}}(-m-n+i))$.
On the other hand, since for every $\ell\in\mathbb{Z}$, $\mathcal{O}_{\tilde{\mathbb{A}}^{2}}(\ell E)=f^{*}\mathcal{O}_{\mathbb{P}^{1}}(-\ell)$
and $f:\tilde{\mathbb{A}}^{2}\rightarrow\mathbb{P}^{1}$ is the total
space of the line bundle $\mathcal{O}_{\mathbb{P}^{1}}(-1)$, it follows
from the projection formula that 
\[
H^{1}(\tilde{\mathbb{A}}^{2},\mathcal{O}_{\tilde{\mathbb{A}}^{2}}(\ell E))\simeq H^{1}(\mathbb{P}^{1},f_{*}\mathcal{O}_{\tilde{\mathbb{A}}^{2}}\otimes\mathcal{O}_{\mathbb{P}^{1}}(-\ell))\simeq\bigoplus_{t\geq0}H^{1}(\mathbb{P}^{1},\mathcal{O}_{\mathbb{P}^{1}}(t-\ell)).
\]
The image of $x^{-m}y^{-n}p\left(x,y\right)$ in $\bigoplus_{s\in\mathbb{Z}}H^{1}(\mathbb{P}^{1},\mathcal{O}_{\mathbb{P}^{1}}(s))$
belongs to $\bigoplus_{t\geq0}H^{1}(\mathbb{P}^{1},\mathcal{O}_{\mathbb{P}^{1}}(t-\ell))$
if and only if $\ell\geq\ell_{0}=m+n-d\geq2$. Given such an $\ell$,
the image $(h_{t})_{t\geq0}\in\bigoplus_{t\geq0}H^{1}(\mathbb{P}^{1},\mathcal{O}_{\mathbb{P}^{1}}(t-\ell))$
of $x^{-m}y^{-n}p\left(x,y\right)$ then defines a unique $M(\ell)$-torsor
$\theta_{\ell}:W(P,\ell)\rightarrow\tilde{\mathbb{A}}^{2}$ whose
restriction over the complement of $E$ is isomorphic to $\rho:P\rightarrow\mathbb{A}_{*}^{2}$
when equipped with the action $\mathbb{G}_{a}$-action induced by
the canonical section of $\mathcal{O}_{\tilde{\mathbb{A}}^{2}}(\ell E)$
with divisor $\ell E$. On the other hand, the restriction of $W|_{E}\rightarrow E$
over $E$ is an $\mathcal{O}_{\mathbb{P}^{1}}(-\ell)$-torsor with
isomorphism class $h_{0}\in H^{1}(\mathbb{P}^{1},\mathcal{O}_{\mathbb{P}^{1}}(-\ell))$.
By definition, $h_{0}$ is non zero if and only if $\ell=\ell_{0}$,
and we conclude from Theorem \ref{prop:Main-Ext-Prop} that $\theta_{\ell_{0}}:W(P,\ell_{0})\rightarrow\tilde{\mathbb{A}}^{2}$
is the unique $\mathbb{G}_{a}$-extension of $\rho:P\rightarrow\tilde{\mathbb{A}}^{2}$
with affine total space. 

\section{Quasi-projective $\mathbb{G}_{a}$-extensions of Type II }

In this section we consider the following subclass of extensions of
Type II of a $\mathbb{G}_{a}$-torsor over a punctured surface.
\begin{defn}
\label{def:Proper-Extension} A $\mathbb{G}_{a}$-extension $\pi:X\rightarrow S$
of a $\mathbb{G}_{a}$-torsor $\rho:P\rightarrow S_{*}$ over a punctured
surface $S_{*}=S\setminus\{o\}$ is said to be a\emph{ quasi-projective}
extension of Type II\emph{ }if it satisfies the following properties

i) $X$ is quasi-projective over $S$ and the $\mathbb{G}_{a,S}$-action
on $X$ is proper, 

ii) $X$ is smooth along $\pi^{-1}(o)$ and $\pi^{-1}(o)_{\mathrm{red}}\simeq\mathbb{A}_{\kappa}^{2}$.
\end{defn}

\begin{example}
\label{exa:ProperExtension} Let $o=V(x,y)$ be a global scheme-theoretic
complete intersection closed point in the smooth locus of a surface
$S$ and let $\rho:P\rightarrow S\setminus\{o\}$ be the $\mathbb{G}_{a}$-torsor
with defining sheaf of ideals $(xv-yu-1)\subset\mathcal{O}_{S}[u,v]$
as in Example \ref{subsec:Homogeneous-Ga-torsors}. Let $\pi_{1}:X_{1}\rightarrow S$
be the affine $S$-scheme with defining sheaf of ideals $(xw-y(yz_{1}+1),xz_{2}-z_{1}(yz_{1}+1),z_{1}w-yz_{2})\subset\mathcal{O}_{S}[z_{1},z_{2},w]$.
The morphism of $S$-schemes $j_{1}:P\rightarrow X_{1}$ defined by
$(x,y,u,v)\mapsto(x,y,u,uv,yu)$ is an open embedding, equivariant
for the $\mathbb{G}_{a,S}$-action on $X_{1}$ associated with the
locally nilpotent $\mathcal{O}_{S}$-derivation $x\partial_{z_{1}}+(2yz_{1}+1)\partial_{z_{2}}+y^{2}\partial_{w}$
of $\pi_{*}\mathcal{O}_{X_{1}}$. The fiber $\pi_{1}^{-1}(o)$ is
isomorphic to $\mathbb{A}_{\kappa}^{2}=\mathrm{Spec}(\kappa[z_{2},w])$
on which the $\mathbb{G}_{a,S}$-action restricts to $\mathbb{G}_{a,\kappa}$-action
by translations associated to the derivation $\partial_{z_{2}}$ of
$\kappa[z_{2},w]$. It is straightforward to check that $X_{1}$ is
smooth along $\pi_{1}^{-1}(o)$. We claim that the geometric quotient
of the $\mathbb{G}_{a,S}$-action on $X_{1}$ is isomorphic to the
complement of a $\kappa$-rational point $o_{1}$ in the blow-up $\tau:\tilde{S}\rightarrow S$
of $o$. Such a surface being in particular separated, the $\mathbb{G}_{a,S}$-action
on $X_{1}$ is proper, implying that $j_{1}:P\hookrightarrow X_{1}$
is a quasi-projective extension of $P$ of Type II.

Indeed, let us identify $\tilde{S}$ with the closed subvariety of
$S\times_{k}\mathrm{Proj}(k[u_{0},u_{1}])$ with equation $xu_{1}-yu_{0}=0$
in such a way that $\tau$ coincides with the restriction of the first
projection. The morphism $f:X_{1}\rightarrow\tilde{S}$ defined by
\[
(x,y,z,u,v)\mapsto((x,y),[x:y])=((x,y),[yz_{1}+1:w])
\]
is $\mathbb{G}_{a}$-invariant and maps $\pi_{1}^{-1}(o)$ dominantly
onto the exceptional divisor $E\simeq\mathrm{pr}_{S}^{-1}(o)\simeq\mathrm{Proj}(\kappa[u_{0},u_{1}])$
of $\tau$. The induced morphism
\[
f|_{\pi^{-1}(o)}:\pi^{-1}(o)=\mathrm{Spec}(\kappa[z_{2},w])\rightarrow E,\quad(z_{2},w)\mapsto[1:w]
\]
factors as the composition of the geometric quotient $\pi_{1}^{-1}(o)\rightarrow\pi_{1}^{-1}(o)/\mathbb{G}_{a,\kappa}\simeq\mathrm{Spec}(\kappa[w])$
with the open immersion $\pi_{1}^{-1}(o)/\mathbb{G}_{a,\kappa}\hookrightarrow E$
of $\pi_{1}^{-1}(o)/\mathbb{G}_{a,\kappa}$ as the complement of the
$\kappa$-rational point $o_{1}=((0,0),[0:1])\in E$. On the other
hand, the composition 
\[
\tau\circ f\circ j_{1}:P\stackrel{\simeq}{\longrightarrow}X_{1}\setminus\pi_{1}^{-1}(o)\rightarrow\tilde{S}\setminus E\stackrel{\simeq}{\longrightarrow}S\setminus\{o\}
\]
coincides with the geometric quotient morphism $\rho:P\rightarrow S\setminus\{o\}$.
So $f:X_{1}\rightarrow\tilde{S}$ factors through a surjective morphism
$q:X_{1}\rightarrow\tilde{S}\setminus\{o_{1}\}$ whose fibers all
consist of precisely one $\mathbb{G}_{a}$-orbit. Since $q$ is a
smooth morphism, $q$ is a $\mathbb{G}_{a}$-torsor which implies
that $X_{1}/\mathbb{G}_{a}\simeq\tilde{S}\setminus\{o_{1}\}$. 
\end{example}

The scheme of the classification of quasi-projective extensions of
Type II of a given $\mathbb{G}_{a}$-torsor $\rho:P\rightarrow S_{*}$
which we give below is as follows: we first construct in \S \ref{subsec:ExtensionFamilies}
families of such extensions, in the form of $\mathbb{G}_{a}$-torsors
$q:X\rightarrow S'$ over quasi-projective $S$-schemes $\tau:S'\rightarrow S$
such that $\tau^{-1}(o)_{\mathrm{red}}$ is isomorphic to $\mathbb{A}_{\kappa}^{1}$,
$S'$ is smooth along $\tau^{-1}(o)$, and $\tau:S'\setminus\tau^{-1}(o)\rightarrow S_{*}$
is an isomorphism. We then show in $\S$ \ref{subsec:Classification}
that for quasi-projective $\mathbb{G}_{a}$-extension $\pi:X\rightarrow S$
of Type II of a given $\mathbb{G}_{a}$-torsor $\rho:P\rightarrow S_{*}$,
the structure morphism $\pi:X\rightarrow S$ factors through a $\mathbb{G}_{a}$-torsor
$q:X\rightarrow S'$ over one of these $S$-schemes $S'$. In the
last subsection, we focus on the special case where $\pi:X\rightarrow S$
has the stronger property of being an affine morphism. 

\subsection{\label{subsec:ExtensionFamilies}A family of $\mathbb{G}_{a}$-extensions
over quasi-projective $S$-schemes}

\indent\newline\indent Let again $(S,o)$ be a pair consisting of
a surface and a closed point $o$ contained in the smooth locus of
$S$, with residue field $\kappa$. We let $\overline{\tau}_{1}:\overline{S}_{1}\rightarrow S$
be the blow-up of $o$, with exceptional divisor $\overline{E}_{1}\simeq\mathbb{P}_{\kappa}^{1}$.
Then for every $n\geq2$, we let $\overline{\tau}_{n,1}:\overline{S}_{n}=\overline{S}_{n}(o_{1},\ldots,o_{n-1})\rightarrow\overline{S}_{1}$
be the scheme obtained from $\overline{S}_{1}$ by performing the
following sequence of blow-ups of $\kappa$-rational points: 

a) The first step $\overline{\tau}_{21}:\overline{S}_{2}(o_{1})\rightarrow\overline{S}_{1}$
is the blow-up of a $\kappa$-rational point $o_{1}\in\overline{E}_{1}$
with exceptional divisor $\overline{E}_{2}\simeq\mathbb{P}_{\kappa}^{1}$, 

b) Then for every $2\leq i\leq n-2$, we let $\overline{\tau}_{i+1,i}:\overline{S}_{i+1}(o_{1},\ldots,o_{i})\rightarrow\overline{S}_{k}(o_{1},\ldots,o_{i-1})$
be the blow-up of a $\kappa$-rational point $o_{i}\in\overline{E}_{i}$,
with exceptional divisor $\overline{E}_{i+1}\simeq\mathbb{P}_{\kappa}^{1}$. 

c) Finally, we let $\overline{\tau}_{n,n-1}:\overline{S}_{n}(o_{1},\ldots,o_{n-1})\rightarrow\overline{S}_{n-1}(o_{1},\ldots,o_{n-2})$
be the blow-up of a $\kappa$-rational point $o_{n-1}\in\overline{E}_{n-1}$
which is a smooth point of the reduced total transform of $\overline{E}_{1}$
by $\overline{\tau}_{1}\circ\cdots\circ\overline{\tau}_{n-1,n-2}$. 

We let $\overline{E}_{n}\simeq\mathbb{P}_{\kappa}^{1}$ be the exceptional
divisor of $\overline{\tau}_{n,n-1}$ and we let 
\[
\overline{\tau}_{n,1}=\overline{\tau}_{2,1}\circ\cdots\circ\overline{\tau}_{n,n-1}:\overline{S}_{n}(o_{1},\ldots,o_{n-1})\rightarrow\overline{S}_{1}.
\]
The inverse image of $o$ in $\overline{S}_{n}(o_{1},\ldots,o_{n-1})$
by $\overline{\tau}_{1}\circ\overline{\tau}_{n,1}$ is a tree of $\kappa$-rational
curves in which $\overline{E}_{n}$ intersects the reduced proper
transform of $\overline{E}_{1}\cup\cdots\cup\overline{E}_{n-1}$ in
$\overline{S}_{n}(o_{1},\ldots,o_{n-1})$ transversally in a unique
$\kappa$-rational point. 

\begin{figure}[ht]
\centering
\begin{tikzpicture}
\draw (0,0) -- (1,0);
\draw (0,0) node[above] {$\overline{E}_1$};
\draw (0,0) node[below] {{\tiny $-1$}};
\draw (0.7,0) node {$\bullet$};
\draw (0.7,0) node[below] {$o_1$};

\draw (2,0) -- (3,0);
\draw (2,0) node[above] {$\overline{E}_1$};
\draw (2,0) node[below] {{\tiny $-2$}};
\draw (2.7,0.1) -- (3.3,-0.7) ;
\draw (3,-0.31) node {$\bullet$}; 
\draw (3,-0.4) node[below] {$o_2$};
\draw (3.4,-0.4) node {{\tiny $-1$}};
\draw (3.5,-0.7) node[below] {$\overline{E}_2$};

\draw (4.5,0) -- (5.5,0);
\draw (4.5,0) node[above] {$\overline{E}_1$};
\draw (4.5,0) node[below] {{\tiny $-2$}};
\draw (5.2,0.1) -- (5.8,-0.7) ;
\draw (5.5,-0.31) node {$\bullet$}; 
\draw (5.5,-0.4) node[below] {$o_3$};
\draw (6,-0.6) node {{\tiny $-2$}};
\draw (6,-0.7) node[below] {$\overline{E}_2$};
\draw (5.3,-0.31) -- (6.3,-0.31);
\draw (6,-0.31) node[above] {{\tiny $-1$}};
\draw (6.6,-0.31) node {$\overline{E}_3$};

\draw (8,0) -- (9,0);
\draw (8,0) node[above] {$\overline{E}_1$};
\draw (8,0) node[below] {{\tiny $-2$}};
\draw (8.7,0.1) -- (9.3,-0.7) ;
\draw (9.5,-0.6) node {{\tiny $-3$}};
\draw (9.5,-0.7) node[below] {$\overline{E}_2$};
\draw (8.8,-0.31) -- (10.4,-0.31);
\draw (10,-0.6) node {{\tiny $-2$}};
\draw (10.3,-0.9) node[below] {$\overline{E}_3$};
\draw (9.5,-0.2) -- (10.1,-1);
\draw (9.3,-0.31) node {$\bullet$}; 
\draw (9.3,-0.31) node[above] {$o_4$};
\draw (10,-0.31) node[above] {{\tiny $-1$}};
\draw (10.8,-0.31) node {$\overline{E}_4$};

\draw (12.2,0) -- (13.2,0);
\draw (12.2,0) node[above] {$\overline{E}_1$};
\draw (12.2,0) node[below] {{\tiny $-2$}};
\draw (12.9,0.1) -- (13.5,-0.7) ;
\draw (13.7,-0.6) node {{\tiny $-3$}};
\draw (13.7,-0.7) node[below] {$\overline{E}_2$};
\draw (13,-0.31) -- (14.6,-0.31);
\draw (14.2,-0.6) node {{\tiny $-2$}};
\draw (14.5,-0.9) node[below] {$\overline{E}_3$};
\draw (13.7,-0.2) -- (14.3,-1);
\draw (14.2,-0.31) node[above] {{\tiny $-2$}};
\draw (15,-0.31) node {$\overline{E}_4$};
\draw (13.5,-0.4) -- (13.5,0.6);
\draw (13.5,0.62) node[above] {$\overline{E}_5$};
\draw (13.7,0.31) node {{\tiny $-1$}};

\draw (0.5, -2) node {$\overline{S}_1$};
\draw [->] (2,-2) -- (1,-2); \draw (1.5,-2) node[above] {$\overline{\tau}_{2,1}$};
\draw (2.7, -2) node {$\overline{S}_2(o_1)$};
\draw [->] (4.5,-2) -- (3.5,-2); \draw (4,-2) node[above] {$\overline{\tau}_{3,2}$};
\draw (5.6, -2) node {$\overline{S}_3(o_1,o_2)$};
\draw [->] (8,-2) -- (7,-2); \draw (7.5,-2) node[above] {$\overline{\tau}_{4,3}$};
\draw (5.6, -2) node {$\overline{S}_3(o_1,o_2)$};
\draw (9.4, -2) node {$\overline{S}_4(o_1,o_2,o_3)$};
\draw [->] (12,-2) -- (11,-2); \draw (11.5,-2) node[above] {$\overline{\tau}_{5,4}$};
\draw (13.6, -2) node {$\overline{S}_5(o_1,o_2,o_3,o_4)$};
\end{tikzpicture}
\caption{The successive total transforms of $\overline{E}_1$ in a possible construction of a surface of the form $\overline{S}_5(o_1,\ldots ,o_4)$ over  a $k$-rational point $o$. The integers indicate the self-intersections of the corresponding curves.} 
\label{fig:DualGraphTypeA} 
\end{figure}
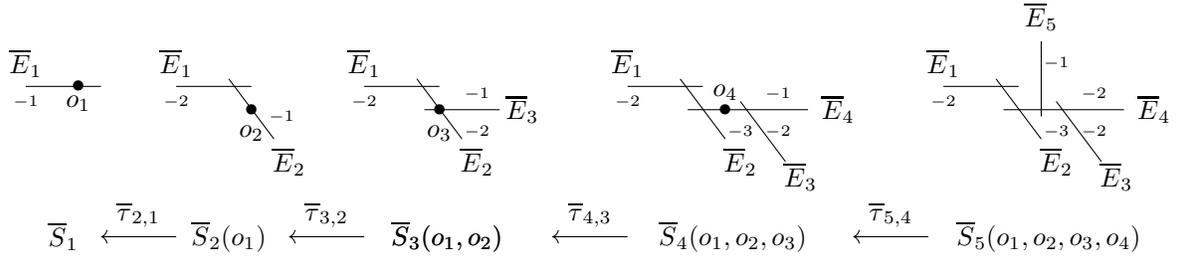
\begin{notation}
\label{nota:Base-scheme} For every $\kappa$-rational point $o_{1}\in\overline{E}_{1}$,
we let $S_{1}(o_{1})=\overline{S}_{1}\setminus\{o_{1}\}$, $E_{1}=\overline{E}_{1}\cap S_{1}\simeq\mathbb{A}_{\kappa}^{1}$
and we let $\tau_{1}:S_{1}(o_{1})\rightarrow S$ be the restriction
of $\overline{\tau}_{1}$. 

For $n\geq2$, we let $S_{n}(o_{1},\ldots,o_{n-1})=\overline{S}_{n}(o_{1},\ldots,o_{n-1})\setminus\overline{E}_{1}\cup\cdots\cup\overline{E}_{n-1}$
and $E_{n}=S_{n}(o_{1},\ldots,o_{n-1})\cap\overline{E}_{n}\simeq\mathbb{A}_{\kappa}^{1}$.
We denote by $\tau_{n,1}:S_{n}(o_{1},\ldots,o_{n-1})\rightarrow\overline{S}_{1}$
the birational morphism induced by $\overline{\tau}_{n,1}$ and we
let $\tau_{n}=\overline{\tau}_{1}\circ\tau_{n,1}:S_{n}(o_{1},\ldots,o_{n-1})\rightarrow S$. 
\end{notation}

The following lemma summarizes some basic properties of the so-constructed
$S$-schemes: 
\begin{lem}
\label{lem:blownup-surfaces-properties} For every $n\geq1$, the
following hold for $S_{n}=S_{n}(o_{1},\ldots,o_{n-1})$:

a) $\tau_{n}:S_{n}\rightarrow S$ is quasi-projective and restricts
to an isomorphism over $S_{*}$ while $\tau_{n}^{-1}(o)_{\mathrm{red}}=E_{n}$, 

b) $S_{n}$ is smooth along $\tau_{n}^{-1}(o)$

c) $\tau_{n}^{*}:\Gamma(S,\mathcal{O}_{S})\rightarrow\Gamma(S_{n},\mathcal{O}_{S_{n}})$
is an isomorphism.

\noindent Moreover for $n\geq2$, the morphism $\tau_{n,1}:S_{n}\rightarrow\overline{S}_{1}$
is affine. 
\end{lem}

\begin{proof}
The first three properties are straightforward consequences of the
construction. For the last one, let $D=\overline{E}_{1}+\sum_{i=2}^{n-1}a_{i}\overline{E}_{i}$
where $a_{i}$ is a sequence of positive rational numbers and let
$m\geq1$ be so that $mD$ is a Cartier divisor on $\overline{S}_{n}$.
Then a direct computation shows that the restriction of $\mathcal{O}_{\overline{S}_{n}}(mD)$
to $\overline{\tau}_{n,1}^{-1}(o_{1})_{\mathrm{red}}=\bigcup_{i=2}^{n}\overline{E}_{i}$
is an ample invertible sheaf provided that the sequence $(a_{i})_{i=2,\ldots,n-1}$
decreases rapidly enough with respect to the distance of $\overline{E}_{i}$
to $\overline{E}_{1}$ in the dual graph of $\overline{E}_{1}\cup\cdots\cup\overline{E}_{n-1}$.
Since $\overline{\tau}_{n,1}$ restricts to an isomorphism over $\overline{S}_{1}\setminus\{o_{1}\}$,
it follows from \cite[Th\'eor\`eme 4.7.1]{EGA3} that $\mathcal{O}_{\overline{S}_{n}}(mD)$
is $\overline{\tau}_{n,1}$-ample on $\overline{S}_{n}$. Since by
definition $\tau_{n,1}$ is the restriction of the projective morphism
$\overline{\tau}_{n,1}:\overline{S}_{n}\rightarrow\overline{S}_{1}$
to $S_{n}=\overline{S}_{n}\setminus\overline{E}_{1}\cup\cdots\cup\overline{E}_{n-1}=\overline{S}_{n}\setminus\mathrm{Supp}(D)$,
we conclude that $\tau_{n,1}$ is an affine morphism. 
\end{proof}
\begin{rem}
By construction, $\tau_{1}^{-1}(o)=E_{1}$ in $S_{1}(o_{1})$, but
for $n\geq2$, we have $\tau_{n}^{-1}(o)=mE_{n}$ for some integer
$m\geq1$ which depends on the sequence of $\kappa$-rational points
$o_{1},\ldots,o_{n-1}$ blown-up to construct $S_{n}(o_{1},\ldots,o_{n-1})$.
For instance, it is straightforward to check that $m=1$ if and only
if for every $i\geq1$, $o_{i}\in\overline{E}_{i}$ is a smooth point
of the reduced total transform of $\overline{E}_{1}$ in $\overline{S}_{i}(o_{1},\ldots,o_{i-1})$. 
\end{rem}

The structure morphism of a $\mathbb{G}_{a}$-torsor being affine,
hence quasi-projective, the total space of any $\mathbb{G}_{a}$-torsor
$q:X\rightarrow S_{n}$ over an $S$-scheme $\tau_{n}:S_{n}=S_{n}(o_{1},\ldots,o_{n})\rightarrow S$
is a quasi-projective $S$-scheme $\pi=\tau_{n}\circ q:X\rightarrow S$
equipped with a proper $\mathbb{G}_{a,S}$-action. Furthermore $\pi^{-1}(o)_{\mathrm{red}}=q^{-1}(E_{n})\simeq E_{n}\times\mathbb{A}_{\kappa}^{1}\simeq\mathbb{A}_{\kappa}^{2}$
and $X$ is smooth along $\pi^{-1}(o)$ as $S_{n}$ is smooth along
$E_{n}$. On the other hand, $\pi:X\rightarrow S$ is by construction
a $\mathbb{G}_{a}$-extension of its restriction $\rho:P\rightarrow S_{n}\setminus E_{n}\simeq S_{*}$
over $S_{n}\setminus E_{n}$, hence is a quasi-projective $\mathbb{G}_{a}$-extension
of $P$ of Type II. The following proposition shows conversely that
every $\mathbb{G}_{a}$-torsor $\rho:P\rightarrow S_{*}$ admits a
quasi-projective $\mathbb{G}_{a}$-extension of Type II  into a $\mathbb{G}_{a}$-torsor
$q:X\rightarrow S_{n}$.
\begin{prop}
\label{prop:Extension-Existence} Let $\rho:P\rightarrow S_{*}$ be
a $\mathbb{G}_{a}$-torsor. Then for every $n\geq1$ and every $S$-scheme
$\tau_{n}:S_{n}(o_{1},\ldots,o_{n-1})\rightarrow S$ as in Notation
\ref{nota:Base-scheme} there exist a $\mathbb{G}_{a}$-torsor $q:X\rightarrow S_{n}(o_{1},\ldots,o_{n-1})$
and an equivariant open embedding $j:P\hookrightarrow X$ such that
in the following diagram \[\begin{tikzcd} P \arrow[hook,r,"j"] \arrow[d,"\rho"'] & X \arrow[d,"q"] \\ S_n(o_1,\ldots , o_{n-1})\setminus E_n \arrow[hook,r] \arrow[d,"\tau_n"',"\wr"] & S_n(o_1,\ldots , o_{n-1}) \arrow[d,"\tau_n"] \\ S_* \arrow[hook,r] & S \end{tikzcd}\]all
squares are cartesian. In particular, $j:P\hookrightarrow X$ is a
quasi-projective $\mathbb{G}_{a}$-extension of $P$ of Type II. 
\[
\]
\end{prop}

\begin{proof}
Letting $S_{n}=S_{n}(o_{1},\ldots,o_{n})$, we have to prove that
every $\mathbb{G}_{a}$-torsor $\rho:P\rightarrow S_{n}\setminus E_{n}\simeq S_{*}$
is the restriction of a $\mathbb{G}_{a}$-torsor $q:X\rightarrow S_{n}$,
or equivalently that the restriction homomorphism $H^{1}(S_{n},\mathcal{O}_{S_{n}})\rightarrow H^{1}(S_{n}\setminus E_{n},\mathcal{O}_{S_{n}\setminus E_{n}})$
is surjective. It is enough to show that there exists a Zariski open
neighborhood $U$ of $E_{n}$ in $S_{n}$ and a $\mathbb{G}_{a}$-torsor
$q:Y\rightarrow U$ such that $Y\mid_{U\setminus E_{n}}\simeq P\mid_{U\setminus E_{n}}$.
Indeed, if so then a $\mathbb{G}_{a}$-torsor $q:X\rightarrow S_{n}$
with the desired property is obtained by gluing $P$ and $Y$ over
$U\setminus E_{n}$ by the isomorphism $Y\mid_{U\setminus E_{n}}\simeq P\mid_{U\setminus E_{n}}$.
In particular, we can replace $S_{n}$ by the inverse image by $\tau_{n}:S_{n}\rightarrow S$
of any Zariski open neighborhood of $o$ in $S$. We can thus assume
from the very beginning that $S=\mathrm{Spec}(A)$ is affine and that
$o=V(f,g)$ is a scheme-theoretic intersection for some $f,g\in A$.
Up to replacing $f$ and $g$ by other generators of the maximal ideal
of $o$ in $A$, we can assume that the proper transform $L_{1}$
in $\overline{\tau}_{1}:\overline{S}_{1}\rightarrow S$ of the curve
$L=V(f)\subset S$ intersects $\overline{E}_{1}$ in $o_{1}$. We
denote by $M_{1}\subset\overline{S}_{1}$ the proper transform of
the curve $M=V(g)\subset S$. 

We first treat the case $n=1$. The open subset $U_{1}=\overline{S}_{1}\setminus L$
of $\overline{S}_{1}$ is then affine and contained in $S_{1}$. Furthermore
$U_{1}\setminus E_{1}=\overline{S}_{1}\setminus\overline{\tau}_{1}^{-1}(L)\simeq S\setminus L$
is also affine. The Mayer-Vietoris long exact sequence of cohomology
of $\mathcal{O}_{S_{1}}$ for the open covering of $S_{1}$ by $S_{1}\setminus E_{1}$
and $U_{1}$ then reads 
\begin{align*}
0\rightarrow & H^{0}(S_{1},\mathcal{O}_{S_{1}})\rightarrow H^{0}(U_{1},\mathcal{O}_{S_{1}})\oplus H^{0}(S_{1}\setminus E_{1},\mathcal{O}_{S_{1}})\rightarrow H^{0}(U_{1}\setminus E_{1},\mathcal{O}_{S_{n}})\rightarrow\cdots\\
\cdots\rightarrow & H^{1}(S_{1},\mathcal{O}_{S_{1}})\rightarrow H^{1}(U_{1},\mathcal{O}_{U_{1}})\oplus H^{1}(S_{1}\setminus E_{1},\mathcal{O}_{S_{1}})\rightarrow H^{1}(U_{1}\setminus E_{1},\mathcal{O}_{U_{1}})\rightarrow\cdots.
\end{align*}
Since $U_{1}\setminus E_{1}$ is affine, $H^{1}(U_{1}\setminus E_{1},\mathcal{O}_{U_{1}})=0$
and so, the homomorphism $H^{1}(S_{1},\mathcal{O}_{S_{1}})\rightarrow H^{1}(S_{1}\setminus E_{1},\mathcal{O}_{S_{1}})$
is surjective as desired. 

In the case where $n\geq2$, the open subset $V_{1}=\overline{S}_{1}\setminus M_{1}$
of $\overline{S}_{1}$ is affine and it contains $o_{1}$ since $M_{1}$
intersects $\overline{E}_{1}$ in a point distinct from $o_{1}$.
Since $\tau_{n,1}:S_{n}\rightarrow\overline{S}_{1}$ is an affine
morphism by Lemma \ref{lem:blownup-surfaces-properties}, $U_{n}=\tau_{n,1}^{-1}(U_{1})$
is an affine open neighborhood of $E_{n}$ in $S_{n}$. By construction,
$S_{n}$ is then covered by the two open subset $U_{n}$ and $S_{n}\setminus E_{n}$
which intersect along the affine open subset $U_{n}\cap S_{n}\setminus E_{n}=U_{n}\setminus E_{n}=\tau_{n,1}^{-1}(\overline{S}_{1}\setminus\overline{\tau}_{1}^{-1}(M))$
of $S_{n}$. The conclusion then follows from the Mayer-Vietoris long
exact sequence of cohomology of $\mathcal{O}_{S_{n}}$ for the open
covering of $S_{n}$ by $S_{n}\setminus E_{n}$ and $U_{n}$. 
\end{proof}

\subsection{\label{subsec:Classification}Classification }

\indent\newline\noindent The following theorem shows that every quasi-projective
$\mathbb{G}_{a}$-extension of Type II of a given $\mathbb{G}_{a}$-torsor
$\rho:P\rightarrow S_{*}$ is isomorphic to one of the schemes $q:X\rightarrow S_{n}$
constructed in $\S$ \ref{subsec:ExtensionFamilies}. 
\begin{thm}
\label{thm:Main-TheoremA2} Let $\rho:P\rightarrow S_{*}$ be a $\mathbb{G}_{a}$-torsor
and let \[\begin{tikzcd}    P \arrow[r,hook,"j"] \arrow[d,"\rho"'] & X \arrow[d,"\pi"] \\ S_* \arrow[r,hook] & S \end{tikzcd}\]
be a quasi-projective $\mathbb{G}_{a}$-extension of $P$ of Type
II. Then there exists an integer $n\geq1$ and a scheme $\tau_{n}:S_{n}(o_{1},\ldots,o_{n-1})\rightarrow S$
such that $X$ is a $\mathbb{G}_{a}$-torsor $q:X\rightarrow S_{n}(o_{1},\ldots,o_{n-1})\simeq X/\mathbb{G}_{a}$
and $\rho:P\rightarrow S_{*}$ coincides with the restriction of $q$
to $S_{n}(o_{1},\ldots,o_{n-1})\setminus E_{n}\simeq S_{*}$. 
\end{thm}

\begin{proof}
Since the $\mathbb{G}_{a,S}$-action on $X$ is proper, the geometric
quotient $X/\mathbb{G}_{a,S}$ exists in the form of a separated algebraic
$S$-space $\delta:X/\mathbb{G}_{a,S}\rightarrow S$. Furthermore,
since by definition of an extension $\pi^{-1}(S_{*})\simeq P$, we
have $\pi^{-1}(S_{*})/\mathbb{G}_{a,S}\simeq P/\mathbb{G}_{a,S}\simeq S_{*}$
and so $\delta$ restricts to an isomorphism over $S_{*}$. On the
other hand, $\pi^{-1}(o)\simeq\mathbb{A}_{\kappa}^{2}$ is equipped
with the induced proper $\mathbb{G}_{a,\kappa}$-action, whose geometric
quotient $\mathbb{A}_{\kappa}^{2}/\mathbb{G}_{a,\kappa}$ is isomorphic
to $\mathbb{A}_{\kappa}^{1}$. It follows from the universal property
of geometric quotient that $\delta^{-1}(o)=\mathbb{A}_{\kappa}^{2}/\mathbb{G}_{a,\kappa}=\mathbb{A}_{\kappa}^{1}$. 

Since $X$ is smooth in a neighborhood of $\pi^{-1}(o)$, $X/\mathbb{G}_{a,S}$
is smooth in neighborhood of $\delta^{-1}(o)$. In particular, $\pi^{-1}(o)$
and $\delta^{-1}(o)$ are Cartier divisors on $X$ and $X/\mathbb{G}_{a,S}$
respectively. Let $\overline{\tau}_{1}:\overline{S}_{1}\rightarrow S$
be the blow-up of $o$. Then by the universal property of blow-ups
\cite[\href{http://stacks.math.columbia.edu/tag/085P} {Tag 085P}]{StackP},
the morphisms $\pi:X\rightarrow S$ and $\delta:X/\mathbb{G}_{a,S}\rightarrow S$
lift to morphisms $\pi_{1}:X\rightarrow\overline{S}_{1}$ and $\delta_{1}:X/\mathbb{G}_{a,S}\rightarrow\overline{S}_{1}$
respectively, and we have a commutative diagram \[\begin{tikzcd} X \arrow[r,"\pi_1"] \arrow[d] & \overline{S}_1 \arrow[d,"\overline{\tau}_1"]\\ X/\mathbb{G}_a \arrow[r,"\delta"] \arrow[ur,"\delta_1"] & S \end{tikzcd}\]Furthermore,
since $\delta:X/\mathbb{G}_{a,S}\rightarrow S$ and $\overline{\tau}_{1}:\overline{S}_{1}\rightarrow S$
are separated, it follows that $\delta_{1}:X/\mathbb{G}_{a,S}\rightarrow\overline{S}_{1}$
is separated. By construction, the image of $\pi^{-1}(o)_{\mathrm{red}}/\mathbb{G}_{a,\kappa}$
by $\delta_{1}$ is contained in $\overline{E}_{1}$. 

If $\delta_{1}$ is not constant on $\pi^{-1}(o)_{\mathrm{red}}/\mathbb{G}_{a,\kappa}$
then $\delta_{1}$ is a separated quasi-finite birational morphism.
Since $\overline{S}_{1}$ is normal, $\delta_{1}$ is thus an open
immersion by virtue of Zariski Main Theorem for algebraic spaces \cite[\href{http://stacks.math.columbia.edu/tag/05W7} {Tag 05W7}]{StackP}.
Since $\pi^{-1}(o)_{\mathrm{red}}/\mathbb{G}_{a,\kappa}\simeq\mathbb{A}_{\kappa}^{1}$,
the only possibility is that $\overline{S}_{1}\setminus\delta_{1}(X/\mathbb{G}_{a,S})$
consists of a unique $\kappa$-rational point $o_{1}\in\overline{E}_{1}$
and $\delta_{1}:X/\mathbb{G}_{a,S}\rightarrow S_{1}(o_{1})=\overline{S}_{1}\setminus\{o_{1}\}$
is an isomorphism. So $\pi_{1}:X\rightarrow S_{1}(o_{1})$ is $\mathbb{G}_{a}$-torsor
whose restriction to $S_{1}(o_{1})\setminus E_{1}\simeq S_{*}$ coincides
with $\rho:P\rightarrow S_{*}$. 

Otherwise, if $\delta_{1}$ is constant on $\pi^{-1}(o)_{\mathrm{red}}/\mathbb{G}_{a,\kappa}$,
then its image consists of a unique $\kappa$-rational point $o_{1}\in\overline{E}_{1}$.
The same argument as above implies that $\pi_{1}:X\rightarrow\overline{S}_{1}$
and $\delta_{1}:X/\mathbb{G}_{a,S}\rightarrow\overline{S}_{1}$ lift
to a $\mathbb{G}_{a,S}$-invariant morphism $\pi_{2}:X\rightarrow\overline{S}_{2}(o_{1})$
and a separated morphism $\delta_{2}:X/\mathbb{G}_{a,S}\rightarrow\overline{S}_{2}(o_{1})$
to the blow-up $\overline{\tau}_{2,1}:\overline{S}_{2}(o_{1})\rightarrow\overline{S}_{1}$
of $\overline{S}_{1}$ at $o_{1}$, with exceptional divisor $\overline{E}_{2}$.
If the restriction of $\delta_{2}$ to $\pi^{-1}(o)_{\mathrm{red}}/\mathbb{G}_{a,\kappa}$
is not constant then $\delta_{2}$ is an open immersion and the image
of $\pi^{-1}(o)_{\mathrm{red}}/\mathbb{G}_{a,\kappa}$ is an open
subset of $\overline{E}_{2}$ isomorphic to $\mathbb{A}_{\kappa}^{1}$.
The only possibility is that $\delta_{2}(\pi^{-1}(o)/\mathbb{G}_{a,\kappa})=\overline{E}_{2}\setminus\overline{E}_{1}$.
Indeed, otherwise $\overline{S}_{2}\setminus\delta_{2}(X/\mathbb{G}_{a,S})$
would consist of the disjoint union of a point in $\overline{E}_{2}\setminus(\overline{E}_{1}\cap\overline{E}_{2})$
and of the curve $\overline{E}_{1}\setminus(\overline{E}_{1}\cap\overline{E}_{2})$
which is not closed in $\overline{S}_{2}$, in contradiction to the
fact that $\delta_{2}$ is an open immersion. Summing up, $\delta_{2}:X/\mathbb{G}_{a,S}\rightarrow S_{2}(o_{1})=\overline{S}_{2}(o_{1})\setminus\overline{E}_{1}$
is an isomorphism mapping $\pi^{-1}(o)_{\mathrm{red}}/\mathbb{G}_{a,\kappa}$
isomorphically onto $E_{2}$. So $\pi_{2}:X\rightarrow S_{2}(o_{1})$
is $\mathbb{G}_{a}$-torsor whose restriction to $S_{2}(o_{1})\setminus E_{2}\simeq S_{*}$
coincides with $\rho:P\rightarrow S_{*}$. 

Otherwise, if $\delta_{2}$ is constant on $\pi^{-1}(o)_{\mathrm{red}}/\mathbb{G}_{a,\kappa}$,
then $\delta_{2}(\pi^{-1}(o)/\mathbb{G}_{a,\kappa})$ is a $\kappa$-rational
point $o_{2}\in\overline{E}_{2}$, and there exists a unique minimal
sequence of blow-ups $\overline{\tau}_{k+1,k}:\overline{S}_{k+1}(o_{1},\ldots,o_{k})\rightarrow\overline{S}_{k}(o_{1},\ldots,o_{k-1})$,
$k=2,\ldots,m-1$ of successive $\kappa$-rational points $o_{k}\in\overline{E}_{k}\subset\overline{S}_{k}(o_{1},\ldots,o_{k-1})$,
with exceptional divisor $\overline{E}_{k+1}\subset\overline{S}_{k+1}(o_{1},\ldots,o_{k})$
such that $\pi_{2}:X\rightarrow\overline{S}_{2}(o_{1})$ and $\delta_{2}:X/\mathbb{G}_{a,S}\rightarrow\overline{S}_{2}(o_{1})$
lift respectively to a $\mathbb{G}_{a,S}$-invariant morphism $\pi_{m}:X\rightarrow\overline{S}_{m}(o_{1},\ldots,o_{m-1})$
and a separated morphism $\delta_{m}:X/\mathbb{G}_{a,S}\rightarrow\overline{S}_{m}(o_{1},\ldots,o_{m-1})$
with the property that the restriction of $\delta_{m}$ to $\pi^{-1}(o)_{\mathrm{red}}/\mathbb{G}_{a,\kappa}$
is non constant. By Zariski Main Theorem \cite[\href{http://stacks.math.columbia.edu/tag/05W7}{Tag 05W7}]{StackP}
again, we conclude that $\delta_{m}$ is an open immersion, mapping
$\pi^{-1}(o)_{\mathrm{red}}/\mathbb{G}_{a,\kappa}\simeq\mathbb{A}_{\kappa}^{1}$
isomorphically onto an open subset of $\overline{E}_{m}\simeq\mathbb{P}_{\kappa}^{1}$.
As in the previous case, the image of $\pi^{-1}(o)_{\mathrm{red}}/\mathbb{G}_{a,\kappa}$
in $\overline{E}_{m}$ must be equal to the complement of the intersection
of $\overline{E}_{m}$ with the proper transform of $\overline{E}_{1}\cup\cdots\cup\overline{E}_{m-1}$
in $\overline{S}_{m}(o_{1},\ldots,o_{m-1})$ since otherwise $\overline{S}_{m}(o_{1},\ldots,o_{m-1})\setminus\delta_{m}(X/\mathbb{G}_{a,S})$
would not be closed in $\overline{S}_{m}(o_{1},\ldots,o_{m-1})$.
Since $\pi^{-1}(o)_{\mathrm{red}}/\mathbb{G}_{a,\kappa}\simeq\mathbb{A}_{\kappa}^{1}$,
it follows that $\overline{E}_{m}$ intersects the proper transform
of $\overline{E}_{1}\cup\cdots\cup\overline{E}_{m-1}$ in a unique
$\kappa$-rational point, implying in turn that $o_{m-1}\in\overline{E}_{m-1}$
is a smooth $\kappa$-rational point of the reduced total transform
$\overline{E}_{1}\cup\cdots\cup\overline{E}_{m-1}$ of $\overline{E}_{1}$
in $\overline{S}_{m-1}(o_{1},\ldots,o_{m-2})$. Summing up, 
\[
\delta_{m}:X/\mathbb{G}_{a,S}\rightarrow\overline{S}_{m}(o_{1},\ldots,o_{m-1})\setminus\overline{E}_{1}\cup\cdots\cup\overline{E}_{m-1}
\]
is an isomorphism with an $S$-scheme of the form $S_{m}(o_{1},\ldots,o_{m-1})$
as constructed in \S \ref{subsec:ExtensionFamilies}, mapping $\pi^{-1}(o)_{\mathrm{red}}/\mathbb{G}_{a,\kappa}$
isomorphically onto $E_{m}=S_{m}(o_{1},\ldots,o_{m-1})\cap\overline{E}_{m}$.
It follows in turn that $\pi_{m}:X\rightarrow S_{m}(o_{1},\ldots,o_{m-1})$
is a $\mathbb{G}_{a}$-torsor whose restriction to $S_{m}(o_{1},\ldots,o_{m-1})\setminus E_{m}\simeq S_{*}$
coincides with $\rho:P\rightarrow S_{*}$. This completes the proof. 
\end{proof}

\subsection{Affine $\mathbb{G}_{a}$-extensions of Type II}

In this subsection, given a $\mathbb{G}_{a}$-torsor $\rho:P\rightarrow S_{*}$
we consider the existence of quasi-projective $\mathbb{G}_{a}$-extensions
of Type II \[\begin{tikzcd}    P \arrow[r,hook,"j"] \arrow[d,"\rho"'] & X \arrow[d,"\pi"] \\ S_* \arrow[r,hook] & S \end{tikzcd}\]
with the additional for which $X$ is affine over $S$. As in the
case of extension to $\mathbb{A}^{1}$-bundles over the blow-up of
$o$ treated in $\S$ \ref{subsec:Ga-blowup-affine}, a necessary
condition for the existence of such extensions is that the restriction
of $P$ over every open neighborhood of the closed point $o$ in $S$
is nontrivial. Indeed, if there exists an affine open neighborhood
$U$ of $o$ over which $P$ is trivial, then $P\simeq U\setminus\{o\}\times\mathbb{A}_{k}^{1}$
is strictly quasi-affine, hence cannot be the complement of a Cartier
divisor $\pi^{-1}(o)$ is any affine $U$-scheme $X|_{U}$. The next
theorem shows that this condition is actually sufficient:
\begin{thm}
\label{thm:Affine-Ga-extensions-A2}Let $\rho:P\rightarrow S_{*}$
be a $\mathbb{G}_{a}$-torsor such that for every open neighborhood
$U$ of $o$ in $S$, the restriction $P\times_{S_{*}}U\rightarrow U\setminus\{o\}$
is non trivial. Then for every $n\geq1$ and every $S$-scheme $\tau_{n}:S_{n}(o_{1},\ldots,o_{n-1})\rightarrow S$
as in Notation \ref{nota:Base-scheme} there exists a quasi-projective
$\mathbb{G}_{a}$-extension of $P$ of Type II into the total space
of a $\mathbb{G}_{a}$-torsor $q:X\rightarrow S_{n}(o_{1},\ldots,o_{n-1})$
for which $\pi=\tau_{n}\circ q:X\rightarrow S$ is an affine morphism. 
\end{thm}

The following example illustrates the strategy of the proof given
below, which consists in constructing such affine extensions $\pi:X\rightarrow S$
by performing a well-chosen equivariant affine modification of extensions
of $\rho:P\rightarrow S_{*}$ into locally trivial $\mathbb{A}^{1}$-bundles
$\theta:W(P)\rightarrow\tilde{S}$ over the blow-up $\tau:\tilde{S}\rightarrow S$
of the point $o$. 
\begin{example}
\label{exa:Standard-Equiv-Aff-Mod} Let again $X_{0}$ and $X_{1}$
be the $\mathbb{G}_{a}$-extensions of $\rho:P=\{xv-yu=1\}\rightarrow S\setminus\{o\}$
considered in Example \ref{subsec:Homogeneous-Ga-torsors} and \ref{exa:ProperExtension}.
Recall that $X_{0}$ and $X_{1}$ are the affine $S$-schemes in $\mathbb{A}_{S}^{3}$
defined respectively by the equations 
\[
X_{0}:\quad\begin{cases}
xr-yq & =0\\
yp-x(q-1) & =0\\
pr-q(q-1) & =0
\end{cases}\qquad\textrm{and}\qquad X_{1}:\quad\begin{cases}
xw-y(yz_{1}+1) & =0\\
xz_{2}-z_{1}(yz_{1}+1) & =0\\
z_{1}w-yz_{2} & =0
\end{cases}
\]
equipped with the $\mathbb{G}_{a,S}$-actions associated with the
locally nilpotent $\mathcal{O}_{S}$-derivations $\partial_{0}=x^{2}\partial_{p}+xy\partial_{q}+y^{2}\partial_{r}$
and $\partial_{1}=x\partial_{z_{1}}+(2yz_{1}+1)\partial_{z_{2}}+y^{2}\partial_{w}$
respectively. 

The morphism $\pi_{0}:X_{0}\rightarrow S$ factors through the structure
morphism $\theta:X_{0}\rightarrow\tilde{S}$ of a torsor under a line
bundle on the blow-up $\tau:\tilde{S}\rightarrow S$ of the origin,
with the property that the restriction of $X_{0}$ to exceptional
divisor $E=\mathbb{P}_{\kappa}^{1}$ of $\tau$ is a nontrivial torsor
under the total space of the line bundle $\mathcal{O}_{\mathbb{P}_{\kappa}^{1}}(-2)$.
The $\mathbb{G}_{a,S}$-action on $X_{0}$ restricts to the trivial
one on $X_{0}|_{E}=\pi_{0}^{-1}(o)$. More precisely, $\partial_{0}$
is a global section of the sheaf $\mathcal{T}_{X_{0}}\otimes\mathcal{O}_{X_{0}}(-2X_{0}|_{E})$
of vector fields on $X_{0}$ that vanish at order $2$ along $X_{0}|_{E}$.
One way to obtain from $X_{0}$ a $\mathbb{G}_{a}$-extension $\pi:X\rightarrow S$
of $\rho:P\rightarrow S\setminus\{o\}$ with fiber $\pi^{-1}(o)_{\mathrm{red}}$
isomorphic to $\mathbb{A}_{\kappa}^{2}$ and a fixed point free action
is thus to perform an equivariant affine modification which simultaneously
replaces $X_{0}|_{E}$ by a copy of $\mathbb{A}_{\kappa}^{2}$ and
decreases the ``fixed point order of $\partial_{0}$ along $X_{0}|_{E}$'',
typically a modification with divisor $D$ equal to $X_{0}|_{E}$
and whose center $Z\subset X_{0}|_{E}$ is supported by a curve isomorphic
to $\mathbb{A}_{\kappa}^{1}$ which is mapped isomorphically onto
its image by the restriction of $\theta$. The birational $S$-morphism
\[
\eta:X_{1}\rightarrow X_{0},\quad(x,y,z_{1},z_{2},w)\mapsto(x,y,xz_{1},yz_{1}+1,w)
\]
is equivariant for the $\mathbb{G}_{a,S}$-actions on $X_{0}$ and
$X_{1}$ and corresponds to an equivariant affine modification of
this type: it restricts to an isomorphism outside the fibers of $\pi_{0}$
and $\pi_{1}$ over $o$, and it contracts $\pi_{1}^{-1}(o)=\mathrm{Spec}(\kappa[z_{2},w])$
onto the curve $\{p=q-1=0\}\subset\pi_{0}^{-1}(o)=\{pr-q(q-1)=0\}$.
This curve is isomorphic to $\mathbb{A}_{\kappa}^{1}=\mathrm{Spec}(\kappa[r])$
and it is mapped by the restriction 
\[
\theta|{}_{\pi_{0}^{-1}(o)}:\pi_{0}^{-1}(o)\simeq\{pr-q(q-1)=0\}\rightarrow E=\mathbb{P}_{\kappa}^{1},\;(p,q,r)\mapsto[p:q-1]=[q:r]
\]
of $\theta$ isomorphically onto the complement of the $\kappa$-rational
point $[0:1]\in\mathbb{P}_{\kappa}^{1}$. 
\end{example}

\begin{proof}[Proof of Theorem \ref{thm:Affine-Ga-extensions-A2}]
By virtue of Theorem \ref{prop:Main-Ext-Prop}, there exists a unique
integer $\ell_{0}\geq2$ such that $\rho:P\rightarrow S_{*}$ is the
restriction of a torsor $\theta_{1}:W_{1}\rightarrow\overline{S}_{1}$
under the line bundle $M_{1}(\ell_{0})=\mathrm{Spec}(\mathrm{Sym}^{\cdot}\mathcal{O}_{\overline{S}_{1}}(-\ell_{0}\overline{E}_{1}))\rightarrow\overline{S}_{1}$
whose total space $W_{1}$ is affine over $\overline{S}_{1}$. We
now treat the case of $S_{1}(o_{1})$ and $S_{n}(o_{1},\ldots,o_{n-1})$,
$n\geq2$ separately.

Given a $\kappa$-rational point $o_{1}\in\overline{E}_{1}$, the
restriction of $W_{1}$ over $E_{1}=\overline{E}_{1}\setminus\{o_{1}\}\simeq\mathbb{A}_{\kappa}^{1}$
is the trivial $\mathbb{A}^{1}$-bundle $E_{1}\times\mathbb{A}_{\kappa}^{1}$.
Since on the other hand the restriction $\theta_{1}|_{\overline{E}_{1}}:W_{1}|_{\overline{E}_{1}}\rightarrow\overline{E}_{1}$
is a non trivial $\mathcal{O}_{\mathbb{P}^{1}}(-\ell_{0})$-torsor
(see Theorem \ref{prop:Main-Ext-Prop}), it follows that for every
section $s:E_{1}\rightarrow W_{1}|_{E_{1}}$ the image $Z_{1}$ of
$E_{1}$ in $W_{1}|_{\overline{E}_{1}}$ is a closed curve isomorphic
to $E_{1}$. Indeed, otherwise if $Z_{1}$ is not closed in $W_{1}|_{\overline{E}_{1}}$
then its closure $\overline{Z}_{1}$ would be a section of $\theta_{1}|_{\overline{E}_{1}}$
in contradiction with the fact that $\theta_{1}|_{\overline{E}_{1}}:W_{1}|_{\overline{E}_{1}}\rightarrow\overline{E}_{1}$
is a non trivial $\mathcal{O}_{\mathbb{P}^{1}}(-\ell_{0})$-torsor.
Let $D_{1}=\theta_{1}^{-1}(\overline{E}_{1})$ and let $\sigma_{1}:W_{1}'\rightarrow W_{1}$
be the affine modification of $W_{1}$ with center $(\mathcal{I}_{Z_{1}},D_{1})$
. By virtue of Lemmas \ref{lem:A1bundle-affmod} and \ref{lem:Torsor-affmod},
$\theta_{1}\circ\sigma_{1}:W_{1}'\rightarrow\overline{S}_{1}$ factors
through a torsor $\theta_{1}':W_{1}'\rightarrow\overline{S}_{1}\setminus\{o_{1}\}=S_{1}(o_{1})$
under the line bundle 
\[
M_{1}'(\ell_{0}-1)=\mathrm{Spec}(\mathrm{Sym}^{\cdot}\mathcal{O}_{S_{1}(o_{1})}((-\ell_{0}+1)E_{1}))\rightarrow S_{1}(o_{1}).
\]
Now since $E_{1}\simeq\mathbb{A}_{\kappa}^{1}$ is affine, the restriction
of $\theta_{1}'$ over $E_{1}\subset S_{1}(o_{1})$ is the trivial
$M_{1}'(\ell_{0}-1)|_{E_{1}}$-torsor. Letting $D_{2}={\theta_{1}'}^{-1}(E_{1})$
and $Z_{2}\subset D_{2}$ be any section of $\theta'_{1}|_{D_{2}}:D_{2}\rightarrow E_{1}$,
the affine modification $\sigma_{2}:W_{2}'\rightarrow W_{1}'$ with
center $(\mathcal{I}_{Z_{2}},D_{2})$ is then an $M_{1}'(\ell_{0}-2)$-torsor
$\theta'_{2}:W_{2}'\rightarrow S_{1}(o_{1})$. Iterating this construction
$\ell_{0}-1$ times, we reach a $\mathbb{G}_{a,S_{1}(o_{1})}$-torsor
$q=\theta_{\ell_{0}+1}':X=W'_{\ell_{0}+1}\rightarrow S_{1}(o_{1})$.
Since $\sigma_{1}:W_{1}'\rightarrow W_{1}$ and each $\sigma_{i}:W_{i}'\rightarrow W_{i-1}'$,
$i\geq2$, restricts to an isomorphism over the complement of $E_{1}$,
the restriction of $q:X\rightarrow S_{1}(o_{1})$ over $S_{1}(o_{1})\setminus E_{1}\simeq S_{*}$
is isomorphic to $\rho:P\rightarrow S_{*}$. Furthermore, since the
morphisms $\sigma_{i}$, $i=1,\ldots,\ell_{0}+1$ are affine and $\overline{\tau}_{1}\circ\theta_{1}:W_{1}\rightarrow S$
is an affine morphism, it follows that 
\[
\tau_{1}\circ q=\overline{\tau}_{1}\circ\theta_{1}\circ\sigma_{1}\circ\cdots\sigma_{\ell_{0}+1}:X\rightarrow S
\]
is an affine morphism. So $q:X\rightarrow S_{1}(o_{1})$ is a $\mathbb{G}_{a}$-extension
of $\rho:P\rightarrow S_{*}$ with the desired property. 

Now suppose that $n\geq2$. It follows from the construction of the
morphism $\tau_{n,1}:S_{n}=S_{n}(o_{1},\ldots,o_{n-1})\rightarrow\overline{S}_{1}$
given in subsection \ref{subsec:ExtensionFamilies} that $\tau_{n,1}^{*}\mathcal{O}_{\overline{S}_{1}}(\ell_{0}\overline{E}_{1})\simeq\mathcal{O}_{S_{n}}(mE_{n})$
for some $m\geq2$. The fiber product $W_{n}=W_{1}\times_{\overline{S}_{1}}S_{n}$
is thus a torsor $\theta_{n}:W_{n}\rightarrow S_{n}$ under the line
bundle 
\[
M_{n}(m)=\mathrm{Spec}(\mathrm{Sym}^{\cdot}\mathcal{O}_{S_{n}}(-mE_{n}))\rightarrow S_{n}
\]
whose restriction to $S_{n}\setminus E_{n}\simeq S_{*}$ is isomorphic
to $\rho:P\rightarrow S_{*}$. Furthermore, since $\tau_{n,1}$ is
an affine morphism by virtue of Lemma \ref{lem:blownup-surfaces-properties},
so is the projection $\mathrm{pr}_{W_{1}}:W_{n}\rightarrow W_{1}$.
Since $\overline{\tau}_{1}\circ\theta_{1}:W_{1}\rightarrow S$ is
an affine morphism, we conclude that $\tau_{n}\circ\theta_{n}=\overline{\tau}_{1}\circ\tau_{n,1}\circ\theta_{n}=\overline{\tau}_{1}\circ\theta\circ\mathrm{pr}_{W_{1}}:W_{n}\rightarrow S$
is an affine morphism as well. Since $E_{n}\simeq\mathbb{A}_{\kappa}^{1}$,
the restriction of $\theta_{n}$ over $E_{n}$ is the trivial $M_{n}(m)|_{E_{n}}$-torsor.
The desired $\mathbb{G}_{a,S_{n}}$-torsor $q:X\rightarrow S_{n}$
extending $\rho:P\rightarrow S_{*}$ is then obtained from $\theta_{n}:W_{n}\rightarrow S_{n}$
by performing a sequence of $m$ successive affine modifications similar
to those applied in the previous case. \end{proof}
\begin{rem}
In the case where $S$ is affine, the total spaces $X$ of the varieties
$q:X\rightarrow S_{n}(o_{1},\ldots,o_{n-1})$ of Theorem \ref{thm:Affine-Ga-extensions-A2}
are all affine. To our knowledge, these are the first instances of
smooth affine threefolds equipped with proper $\mathbb{G}_{a}$-actions
whose geometric quotients are smooth quasi-projective surfaces which
are not quasi-affine. 
\end{rem}

We do not know in general if under the conditions of Theorem \ref{thm:Affine-Ga-extensions-A2}
every quasi-projective $\mathbb{G}_{a}$-extensions of $P$ of Type
II into the total space of a $\mathbb{G}_{a}$-torsor $q:X\rightarrow S_{n}(o_{1},\ldots,o_{n-1})$
has the property that $\pi=\tau_{n}\circ q:X\rightarrow S$ is an
affine morphism. In particular, we ask the following:
\begin{question}
Is the total space $X$ of a quasi-projective $\mathbb{G}_{a}$-extension
$\pi:X\rightarrow\mathbb{A}^{2}$ of $\rho=\mathrm{pr}_{x,y}:\mathrm{SL}_{2}=\{xv-yu=1\}\rightarrow\mathbb{A}_{*}^{2}$
of Type II always an affine variety ? 
\end{question}

\subsection{Examples}

In the next paragraphs, we construct two countable families of quasi-projective
$\mathbb{G}_{a}$-extensions of the $\mathbb{G}_{a}$-torsor $\mathrm{SL}_{2}\rightarrow\mathrm{SL}_{2}/\mathbb{G}_{a}\simeq\mathbb{A}^{2}\setminus\{(0,0)\}$
of Type II with affine total spaces. As a consequence of \cite[Section 3]{He15},
for any nontrivial $\mathbb{G}_{a}$-torsor $\rho:P\rightarrow S_{*}$
over a local punctured surface $S_{*}$, these provide, by suitable
base changes, families of examples of $\mathbb{G}_{a}$-extensions
of $P$ whose total spaces are all affine over $S$. 

\subsubsection{\label{subsec:TypeA} A family of $\mathbb{G}_{a}$-extensions of
$\mathrm{SL}_{2}$ of ``Type II-A'' }

Let $S=\mathbb{A}^{2}=\mathrm{Spec}(k[x,y_{0}])$ and let $X_{n}\subset\mathbb{A}_{S}^{n+2}=\mathrm{Spec}(k[x,y_{0}][z_{1},z_{2},y_{1},\ldots,y_{n}])$,
$n\geq1$, be the smooth threefold defined by the system of equations
\[
\begin{cases}
y_{i}y_{j}-y_{k}y_{\ell} & =0\quad i,j,k,\ell=0,\ldots,n,\;i+j=k+\ell\\
z_{2}y_{i}-z_{1}y_{i+1} & =0\quad i=0,\ldots,n-1\\
xy_{i+1}-y_{i}(y_{0}z_{1}+1) & =0\quad i=0,\ldots,n-1\\
xz_{2}-z_{1}(y_{0}z_{1}+1) & =0.
\end{cases}
\]
The threefold $X_{n}$ can be endowed with a fixed point free $\mathbb{G}_{a,S}$-action
induced by the locally nilpotent $k[x,y_{0}]$-derivation 
\[
x\partial_{z_{1}}+(2y_{0}z_{1}+1)\partial_{z_{2}}+\sum_{i=1}^{n}iy_{0}y_{i-1}\partial_{y_{i}}
\]
of its coordinate ring. The scheme-theoretic fiber over $o=\{(0,0)\}$
of the $\mathbb{G}_{a}$-invariant morphism $\pi_{n}=\mathrm{pr}_{x,y,}:X_{n}\rightarrow S$
is isomorphic $\mathbb{A}^{2}=\mathrm{Spec}(k[z_{2},y_{n}])$, on
which the induced $\mathbb{G}_{a}$-action is a translation induced
by the derivation $\partial_{z_{2}}$ of $k[z_{2},y_{n}]$. On the
other hand, the morphism $j:\mathrm{SL}_{2}=\left\{ xv-y_{0}u=1\right\} \rightarrow X_{n}$
defined by 
\[
(x,y,u,v)\mapsto(x,u,uv,y,yv,yv^{2},\ldots,yv^{n})
\]
is an equivariant open embedding of $\mathrm{SL}_{2}$ equipped with
the $\mathbb{G}_{a}$-action induced by the locally nilpotent derivation
$x\partial_{u}+y_{0}\partial_{v}$ of its coordinate ring into $X_{n}$
with image equal to $\pi^{-1}(\mathbb{A}^{2}\setminus\{o\})$. So
$j:\mathrm{SL}_{2}\hookrightarrow X_{n}$ is a quasi-projective $\mathbb{G}_{a}$-extension
of $\mathrm{SL}_{2}$ into the affine variety $X_{n}$, with $\pi_{n}^{-1}(o)\simeq\mathbb{A}_{k}^{2}$. 

The restrictions of the projection $\mathbb{A}_{S}^{n+3}\rightarrow\mathbb{A}_{S}^{n+2}$
onto the first $n+2$ variables induce a sequence of $\mathbb{G}_{a}$-equivariant
birational morphisms $\sigma_{n+1,n}:X_{n+1}\rightarrow X_{n}$. The
threefolds $X_{n}$ thus form a countable tower of $\mathbb{G}_{a}$-equivariant
affine modifications of $X_{1}$. It follows from Example \ref{exa:ProperExtension}
that $X_{1}$ is a quasi-projective extension of $\mathrm{SL}_{2}$
of Type II with geometric quotient isomorphic to a quasi-projective
surface of the form $S_{1}(o_{1})$. More generally, we have the following
result. 
\begin{prop}
For every $n\geq2$, the morphism $j:\mathrm{SL}_{2}\hookrightarrow X_{n}$
is a quasi-projective $\mathbb{G}_{a}$-extension of Type II. The
geometric quotient $X_{n}/\mathbb{G}_{a}$ is isomorphic to a quasi-projective
surface $S_{n}=S_{n}(o_{1},\ldots,o_{n})$ as in $\S$ \ref{subsec:ExtensionFamilies}
for which $\overline{S}_{n}(o_{1},\ldots,o_{n-1})\setminus S_{n}$
consists of a chain of $n-1$ smooth rational curves with self-intersection
$-2$, i.e. the exceptional set of the minimal resolution of a surface
singularity of type $A_{n-1}$. 
\end{prop}

\begin{proof}
To see this, we consider the following sequence of blow-ups: the first
one $\overline{\tau}_{1}:\overline{S}_{1}\rightarrow U_{0}=\mathbb{A}^{2}$
is the blow-up of the origin, with exceptional divisor $\overline{E}_{1}$,
and we let $U_{1}\simeq\mathbb{A}^{2}=\mathrm{Spec}(k[x,w_{1}])$
be the affine chart of $\overline{S}_{1}$ on which $\overline{\tau}_{1}:\overline{S}_{1}\rightarrow\mathbb{A}^{2}$
is given by $(x,w_{1})\mapsto(x,xw_{1})$. Then we let $\overline{\tau}_{2,1}:\overline{S}_{2}(o_{1})\rightarrow\overline{S}_{1}$
be the blow-up of the point $o_{1}=(0,0)\in U_{1}\subset\overline{S}_{1}$
with exceptional divisor $\overline{E}_{2}$, and we let $U_{2}\simeq\mathbb{A}^{2}=\mathrm{Spec}(k[x,w_{2}])$
be the affine chart of $\overline{S}_{2}(o_{1})$ on which the restriction
of $\overline{\tau}_{2,1}:\overline{S}_{2}(o_{1})\rightarrow\overline{S}_{1}$
coincides with the morphism $U_{2}\rightarrow U_{1}$, $(x,w_{2})\mapsto(x,xw_{2})$.
For every $2<m\leq n$, we define by induction the blow-up 
\[
\overline{\tau}_{m,m-1}:\overline{S}_{m}(o_{1},\ldots,o_{m-1})\rightarrow\overline{S}_{m-1}(o_{1},\ldots,o_{m-2})
\]
of the point $o_{m-1}=(0,0)\in U_{m-1}\subset\overline{S}_{m-1}(o_{1},\ldots,o_{m-2})$
with exceptional divisor $\overline{E}_{m}$ and we let $U_{m}\simeq\mathbb{A}^{2}=\mathrm{Spec}(k[x,w_{m}])$
be the affine chart of $\overline{S}_{m}(o_{1},\ldots,o_{m-1})$ on
which the restriction of $\overline{\tau}_{m,m-1}$ coincides with
the morphism $U_{m}\rightarrow U_{m-1}$, $(x,w_{m})\mapsto(x,xw_{m})$.
By construction, we have a commutative diagram \[\xymatrix@C8ex{\overline{S}_n(o_1,\ldots,o_{n-1}) \ar[r]^{\overline{\tau}_{n,n-1}} & \overline{S}_{n-1}(o_1,\ldots,o_{n-2}) \ar[r]^-{\overline{\tau}_{n-1,n-2}}  & \cdots &  \ar[r]^{\overline{\tau}_{2,1}} & \overline{S}_{1} \ar[r]^{\overline{\tau}_1} & \mathbb{A}^2 \\ U_n \ar[u] \ar[r] & U_{n-1} \ar[u] \ar[r] & \cdots & \ar[r] & U_1 \ar[u] \ar[r] & \mathbb{A}^2=U_0. \ar@{=}[u]} \]  

The total transform of $\overline{E}_{1}$ in $\overline{S}_{n}(o_{1},\ldots,o_{n-1})$
is a chain $\overline{E}_{1}\cup\overline{E}_{2}\cup\cdots\cup\overline{E}_{n-1}\cup\overline{E}_{n}$
is a chain formed of $n-1$ curves with self-intersection $-2$ and
the curve $\overline{E}_{n}$ which has self-intersection $-1$. 

\begin{figure}[ht]
\centering
\begin{tikzpicture}
\draw (0,0) -- (1,0) --(1.5,0);
\draw [dashed](1.5,0) -- (7.5,0);
\draw (7.5,0) -- (8,0) -- (9,0);
\draw (0,0) node {$\bullet$}; 
\draw (0,0) node[above] {$\overline{E}_1$}; 
\draw (0,0) node[below] {$-2$};
\draw (1,0) node {$\bullet$};
\draw (1,0) node[above] {$\overline{E}_2$}; 
\draw (1,0) node[below] {$-2$};
\draw (8,0) node {$\bullet$};
\draw (8,0) node[above] {$\overline{E}_{n-1}$}; 
\draw (8,0) node[below] {$-2$};
\draw (9,0) node {$\bullet$};
\draw (9,0) node[above] {$\overline{E}_{n}$}; 
\draw (9,0) node[below] {$-1$};
\end{tikzpicture}
\caption{Dual graph of the total transform of $\overline{E}_1$ in $\overline{S}_n(o_1,\ldots ,o_n)$.} 
\label{fig:DualGraphTypeA} 
\end{figure}

The morphism $\pi:X_{n}\rightarrow S$ lifts to a morphism $\pi_{1}:X_{n}\rightarrow\overline{S}_{1}$
defined by 
\[
(x,z_{1},z_{2},y_{0},y_{1},\ldots,y_{n})\mapsto((x,y_{0}),[x:y_{0}])=((x,y),[y_{0}z_{1}+1:y_{1}]).
\]
This morphism contracts $\pi^{-1}(o)$ onto the point $o_{1}=((0,0),[1:0])$
of the exceptional divisor $\overline{E}_{1}$ of $\overline{\tau}_{1}$.
The induced rational map $\pi_{1}:X_{n}\dashrightarrow U_{1}$ is
given by 
\[
(x,z_{1},z_{2},y_{0},y_{1},\ldots,y_{n})\mapsto(x,\frac{y_{1}}{y_{0}z_{1}+1})
\]
and it contracts $\pi^{-1}(o)$ onto the origin $o_{1}=(0,0)$. So
$\pi_{1}$ lifts to a morphism $\pi_{2}:X_{n}\rightarrow\overline{S}_{2}(o_{1})$,
and with our choice of charts, the induced rational map $\pi_{2}:X_{n}\dashrightarrow U_{2}$
is given by 
\[
(x,z_{1},z_{2},y_{0},y_{1},\ldots,y_{n})\mapsto(x,\frac{y_{2}}{(y_{0}z_{1}+1)^{2}})
\]
If $n=2$ then the image of $\pi^{-1}(o)=\mathrm{Spec}(k[z_{2},y_{2}])$
by $\pi_{2}$ is equal to $\overline{E}_{2}\cap U_{2}$ and $\pi_{2}^{-1}(\overline{E}_{2}\cap U_{2})$
is equivariantly isomorphic to $(\overline{E}_{2}\cap U_{2})\times\mathrm{Spec}(k[z_{2}])$
on which $\mathbb{G}_{a}$ acts by translations on the second factor.
So $\pi_{2}:X_{n}\rightarrow\overline{S}_{2}(o_{1})$ factors through
a $\mathbb{G}_{a}$-bundle $q_{2}:X_{2}\rightarrow S_{2}(o_{1})=\overline{S}_{2}(o_{1})\setminus\overline{E}_{1}$
and $X_{2}/\mathbb{G}_{a}\simeq S_{2}(o_{1})$. Otherwise, if $n>2$
then $\pi_{2}$ contracts $\pi^{-1}(o)$ onto the point $o_{2}=(0,0)\in\overline{E}_{2}\cap U_{2}\subset\overline{S}_{2}(o_{1})$.
So $\pi_{2}:X_{n}\rightarrow\overline{S}_{2}(o_{1})$ lifts to a morphism
$\pi_{3}:X_{n}\rightarrow\overline{S}_{3}(o_{1},o_{2})$. With our
choice of charts, for each $2<m<n$, the induced rational map $\pi_{m}:X_{n}\dashrightarrow U_{m}$
is given by 
\[
(x,z_{1},z_{2},y_{0},y_{1},\ldots,y_{n})\mapsto(x,\frac{y_{m}}{(y_{0}z_{1}+1)^{m}})
\]
hence contracts $\pi^{-1}(o)$ onto the point $o_{m}=(0,0)\in U_{m}\subset\overline{S}_{m}(o_{1},\ldots,o_{m-1})$.
It thus lifts to a morphism $\pi_{m}:X_{n}\rightarrow\overline{S}_{m}(o_{1},\ldots,o_{m-1})$.
At the last step, the image of $\pi^{-1}(o)=\mathrm{Spec}(k[z_{2},y_{n}])$
by the rational map $\pi_{n}:X_{n}\dashrightarrow U_{n}$ induced
by $\pi_{n}:X_{n}\rightarrow\overline{S}_{n}(o_{1},\ldots,o_{n-1})$
is equal to $\overline{E}_{n}\cap U_{n}$, and we conclude as above
that $\pi_{n}:X_{n}\rightarrow\overline{S}_{n}(o_{1},\ldots,o_{n-1})$
factors through a $\mathbb{G}_{a}$-bundle 
\[
q_{n}:X_{n}\rightarrow S_{n}(o_{1},\ldots,o_{n-1})=\overline{S}_{n}(o_{1},\ldots,o_{n-1})\setminus(\overline{E}_{1}\cup\cdots\cup\overline{E}_{n-1}),
\]
hence that $X_{n}/\mathbb{G}_{a}$ is isomorphic to the quasi-projective
surface $S_{n}(o_{1},\ldots,o_{n-1})$. 
\end{proof}

\subsubsection{A family of $\mathbb{G}_{a}$-extensions of $\mathrm{SL}_{2}$ of
``Type II-D''}

To conclude this section, we present as an illustration of the proof
of Theorem \ref{thm:Affine-Ga-extensions-A2} another countable family
of quasi-projective $\mathbb{G}_{a}$-extensions of $\mathrm{SL}_{2}$
of Type II with  affine total spaces. 

Let again $\overline{\tau}_{1}:\overline{S}_{1}\rightarrow S=\mathbb{A}^{2}$
be the blow-up of the origin $o=\{(0,0)\}$ in $\mathbb{A}^{2}=\mathrm{Spec}(k[x,y])$
with exceptional divisor $\overline{E}_{1}\simeq\mathbb{P}^{1}$,
identified with closed subvariety of $\mathbb{A}^{2}\times\mathbb{P}_{[w_{0}:w_{1}]}^{1}$
with equation $xw_{1}-yw_{0}=0$ in such a way that $\tau$ coincides
with the restriction of the first projection. The second projection
identifies $\overline{S}_{1}$ with the total space $p:\overline{S}_{1}\rightarrow\mathbb{P}^{1}$
of the invertible sheaf $\mathcal{O}_{\mathbb{P}^{1}}(-1)$. We fix
trivializations $p^{-1}(U_{\infty})=\mathrm{Spec}(k[z_{\infty}][u_{\infty}])$
and $p^{-1}(U_{0})=\mathrm{Spec}(k[z_{0}][u_{0}])$ over the open
subsets $U_{\infty}=\mathbb{P}^{1}\setminus\{[0:1]\}=\mathrm{Spec}(k[z_{\infty}])$
and $U_{0}=\mathbb{P}^{1}\setminus\{[1:0]\}=\mathrm{Spec}(k[z_{0}])$
in such a way that the gluing of $p^{-1}(U_{\infty})$ and $p^{-1}(U_{0})$
over $U_{0}\cap U_{\infty}$ is given by the isomorphism $(z_{0},u_{0})\mapsto(z_{\infty},u_{\infty})=(z_{0}^{-1},z_{0}u_{0})$.

For every $n\geq1$, we let $S_{2n+3,0}=\mathrm{Spec}(k[z_{0},u_{0}^{\pm1}])$,
\[
S_{2n+3,\infty}=\mathrm{Spec}(k[z_{\infty},u_{\infty},v_{\infty}]/(u_{\infty}^{n}v_{\infty}-z_{\infty}^{2}-u_{\infty})),
\]
and we let $S_{2n+3}$ be the surface obtained by gluing $S_{2n+3,0}$
and $S_{2n+3,\infty}$ along the open subsets $S_{2n+3,0}\setminus\{z_{0}=0\}$
and $S_{2n+3,\infty}\setminus\{z_{\infty}=u_{\infty}=0\}$ by the
isomorphism 
\[
(z_{0},u_{0})\mapsto(z_{\infty},u_{\infty},v_{\infty})=(z_{0}^{-1},z_{0}u_{0},(z_{0}u_{0})^{-n}(z_{0}^{-2}+z_{0}u_{0})).
\]
The canonical open immersion $S_{2n+3,0}\hookrightarrow p^{-1}(U_{0})$
and the projection $\mathrm{pr}_{z_{\infty},u_{\infty}}:S_{2n+3,\infty}\rightarrow p^{-1}(U_{\infty})$
glue to a global birational affine morphism $\tau_{2n+3,1}:S_{2n+3}\rightarrow\overline{S}_{1}$
restricting to an isomorphism $S_{2n+3}\setminus\{z_{\infty}=u_{\infty}=0\}\rightarrow\overline{S}_{1}\setminus\overline{E}_{1}$
where we identified the closed subset $E_{2n+3}=\{z_{\infty}=u_{\infty}=0\}\simeq\mathrm{Spec}(k[v_{\infty}])$
of $S_{2n+3,\infty}$ with its image in $S_{2n+3}$. We leave to the
reader to check that with the notation of $\S$ \ref{subsec:ExtensionFamilies},
$S_{2n+3}=S_{2n+3}(o_{1},\ldots,o_{2n+2})$ for a surface $\overline{\tau}_{2n+3,1}:\overline{S}_{2n+3,1}(o_{1},\ldots,o_{2n+2})\rightarrow\overline{S}_{1}$
obtained by first blowing-up the point $o_{1}=(0,0)\in p^{-1}(U_{\infty})$
with exceptional divisor $\overline{E}_{2}$, then the point $o_{2}=\overline{E}_{1}\cap\overline{E}_{2}$
with exceptional divisor $\overline{E}_{3}$, then a point $o_{3}\in\overline{E}_{3}\setminus(\overline{E}_{1}\cup\overline{E}_{2})$
with exceptional divisor $\overline{E}_{4}$ and then a sequence of
point $o_{i}\in\overline{E}_{i}\setminus\overline{E}_{i-1}$ with
exceptional divisor $\overline{E}_{i+1}$, $i=5,\ldots,2n+2$ in such
a way that the total transform of $\overline{E}_{1}$ in $\overline{S}_{2n+3,1}$
is a tree depicted in Figure \ref{fig:DualGraphTypeD}. Letting $\tau_{2n+3}=\overline{\tau}_{1}\circ\tau_{2n+3,1}:S_{2n+3}\rightarrow\mathbb{A}^{2}$,
we have $\tau_{2n+3}^{-1}(o)_{\mathrm{red}}=E_{2n+3}\simeq\mathbb{A}^{1}$
and $\tau_{2n+3}^{*}(o)=2E_{2n+3}$. 

\begin{figure}[ht]
\centering
\begin{tikzpicture}
\draw (-0.7,0.7) -- (0,0);
\draw (-0.7,-0.7) -- (0,0);
\draw (0,0) -- (1,0) --(1.5,0);
\draw [dashed](1.5,0) -- (7.5,0);
\draw (7.5,0) -- (8,0) -- (9,0);
\draw (-0.7,0.7) node {$\bullet$};
\draw (-0.7,0.7) node[above] {$\overline{E}_1$}; 
\draw (-0.8,0.6) node[below] {$-3$};

\draw (-0.7,-0.7) node {$\bullet$};
\draw (-0.7,-0.6) node[above] {$\overline{E}_2$}; 
\draw (-0.7,-0.7) node[below] {$-2$};

\draw (0,0) node {$\bullet$}; 
\draw (.2,0) node[above] {$\overline{E}_3$}; 
\draw (0.2,0) node[below] {$-2$};

\draw (1,0) node {$\bullet$}; 
\draw (1,0) node[above] {$\overline{E}_4$}; 
\draw (1,0) node[below] {$-2$};

\draw (8,0) node {$\bullet$};
\draw (8,0) node[above] {$\overline{E}_{2n+2}$}; 
\draw (8,0) node[below] {$-2$};

\draw (9,0) node {$\bullet$};
\draw (9,0) node[above] {$\overline{E}_{2n+3}$}; 
\draw (9,0) node[below] {$-1$};
\end{tikzpicture}
\caption{Dual graph of the total transform of $\overline{E}_1$ in $\overline{S}_{2n+3}(o_{1},\ldots,o_{2n+2})$.} 
\label{fig:DualGraphTypeD} 
\end{figure}

Now we let $q:X_{2n+3}\rightarrow S_{2n+3}$ be the $\mathbb{G}_{a}$-bundle
defined as the gluing of the trivial $\mathbb{G}_{a}$-bundles $X_{2n+3,0}=S_{2n+3,0}\times\mathrm{Spec}(k[t_{0}])$
and $X_{2n+3,\infty}=S_{2n+3,\infty}\times\mathrm{Spec}(k[t_{\infty}])$
over $S_{2n+3,0}$ and $S_{2n+3,\infty}$ respectively along the open
subsets $X_{2n+3,0}\setminus\{z_{0}=0\}$ and $X_{2n+3,\infty}\setminus\{z_{\infty}=u_{\infty}=0\}$
by the $\mathbb{G}_{a}$-equivariant isomorphism 
\[
(z_{0},u_{0},t_{0})\mapsto(z_{\infty},u_{\infty},v_{\infty},t_{\infty})=(z_{0}^{-1},z_{0}u_{0},(z_{0}u_{0})^{-n}(z_{0}^{-2}+z_{0}u_{0}),t_{0}+z_{0}^{-1}u_{0}^{-2}).
\]
Let $\pi_{2n+3}=\overline{\tau}_{1}\circ\tau_{2n+3,1}\circ q:X_{2n+3}\rightarrow\mathbb{A}^{2}$. 
\begin{prop}
For every $n\geq1$, the variety $X_{2n+3}$ is affine and there exists
a $\mathbb{G}_{a}$-equivariant open embedding $j:\mathrm{SL}_{2}\hookrightarrow X_{2n+3}$
which makes $\pi_{2n+3}:X_{2n+3}\rightarrow\mathbb{A}^{2}$ a quasi-projective
$\mathbb{G}_{a}$-extension of $\mathrm{SL}_{2}$ of Type II, with
fiber $\pi_{2n+3}^{-1}(o)$ isomorphic to $\mathbb{A}^{2}$ of multiplicity
two, and geometric quotient $X_{2n+3}/\mathbb{G}_{a}\simeq S_{2n+3}$.
\end{prop}

\begin{proof}
Let $j_{1}:\mathrm{SL}_{2}\hookrightarrow W=W(\mathrm{SL}_{2},2)$
be the $\mathbb{G}_{a}$-extension of $\mathrm{SL}_{2}$ into a locally
trivial $\mathbb{A}^{1}$-bundle $\theta:W\rightarrow\overline{S}_{1}$
with affine total space constructed in Example \ref{subsec:Homogeneous-Ga-torsors}.
Recall that the image of $j_{1}$ coincides with the restriction of
$\theta$ to $\overline{S}_{1}\setminus\overline{E}_{1}=\mathbb{A}^{2}\setminus\{o\}$.
With our choice of coordinates, the open subsets $W_{0}=\theta^{-1}(q^{-1}(U_{0}))$
and $W_{\infty}=\theta^{-1}(q^{-1}(U_{\infty}))$ of $W$ are respectively
isomorphic to $p^{-1}(U_{0})\times\mathrm{Spec}(k[w_{0}])$ and $p^{-1}(U_{\infty})\times\mathrm{Spec}(k[w_{\infty}])$
glued over $U_{0}\cap U_{\infty}$ by the isomorphism 
\[
(z_{0},u_{0},w_{0})\mapsto(z_{\infty},u_{\infty},w_{\infty})=(z_{0}^{-1},z_{0}u_{0},z_{0}^{2}w_{0}+z_{0}).
\]
The $\mathbb{G}_{a}$-action on $W_{0}$ and $W_{\infty}$ are given
respectively by $\alpha\cdot(z_{0},u_{0},w_{0})=(z_{0},u_{0},w_{0}+\alpha u_{0}^{2})$
and $\alpha\cdot(z_{\infty},u_{\infty},w_{\infty})=(z_{\infty},u_{\infty},w_{\infty}+\alpha u_{\infty}^{2})$. 

Let $W'=W\times_{\overline{S}_{1}}S_{2n+3}$, equipped with the natural
lift of the $\mathbb{G}_{a}$-action on $W$. Since $\tau_{2n+3,1}:S_{2n+3}\rightarrow\overline{S}_{1}$
restricts to an isomorphism over $\overline{S}_{1}\setminus\overline{E}_{1}$,
the composition $j'=\tau_{2n+3,1}^{-1}\circ j_{1}:\mathrm{SL}_{2}\rightarrow W'$
is a $\mathbb{G}_{a}$-equivariant open embedding. Furthermore, since
$W$ is affine and $\tau_{2n+3,1}$ is an affine morphism, it follows
that $W'$ is affine. By construction, $W'$ is covered by the two
open subsets 
\[
\begin{cases}
W'_{0}=W\times_{p^{-1}(U_{0})}S_{2n+3,0} & \simeq S_{2n+3,0}\times\mathrm{Spec}(k[w_{0}])\\
W'_{\infty}=W\times_{p^{-1}(U_{\infty})}S_{2n+3,\infty} & \simeq S_{2n+3,\infty}\times\mathrm{Spec}(k[w_{\infty}]).
\end{cases}
\]
The local $\mathbb{G}_{a}$-equivariant morphisms 
\[
\begin{cases}
\beta_{0}:X_{2n+3,0}=S_{2n+3,0}\times\mathrm{Spec}(k[t_{0}])\rightarrow W_{0}'\\
\beta_{\infty}:X_{2n+3,\infty}=S_{2n+3,\infty}\times\mathrm{Spec}(k[t_{\infty}])\rightarrow W_{\infty}'
\end{cases}
\]
 of schemes over $S_{2n+1,0}$ and $S_{2n+3,\infty}$ respectively
defined by $t_{0}\mapsto w_{0}=u_{0}^{2}t_{0}$ and $t_{\infty}\mapsto w_{\infty}=u_{\infty}^{2}t_{\infty}$
glue to a global $\mathbb{G}_{a}$-equivariant birational affine morphism
$\beta:X_{2n+3}\rightarrow W'$, restricting to an isomorphism over
$S_{2n+3}\setminus E_{2n+3}\simeq\mathbb{A}^{2}\setminus\{o\}$. Summing
up, $X_{2n+3}$ is affine over $W'$ hence affine, and the composition
$\beta^{-1}\circ j':\mathrm{SL}_{2}\hookrightarrow X_{2n+3}$ is a
$\mathbb{G}_{a}$-equivariant open embedding which realizes $\pi:X_{2n+3}\rightarrow\mathbb{A}^{2}$
as a $\mathbb{G}_{a}$-extension of $\mathrm{SL}_{2}$ of Type II
with affine total space. By construction, $\pi_{2n+3}^{-1}(o)=q^{-1}(2E_{2n+3})$
is isomorphic to $\mathbb{A}^{2}$, with multiplicity two, while the
geometric quotient $X_{2n+3}/\mathbb{G}_{a}$ is isomorphic to $S_{2n+3}$. 
\end{proof}
\begin{rem}
For every $n\geq1$, the birational morphism $S_{2(n+1)+3,\infty}\rightarrow S_{2n+3,\infty}$,
$(z_{\infty},u_{\infty},v_{\infty})\mapsto(z_{\infty},u_{\infty},u_{\infty}v_{\infty})$
extends to a birational morphism $S_{2(n+1)+3}\rightarrow S_{2n+3}$
which lifts in turn in a unique way to a $\mathbb{G}_{a}$-equivariant
birational morphism $\gamma_{n+1,n}:X_{2(n+1)+3}\rightarrow X_{2n+3}$.
So in a similar way as for the family constructed in $\S$ \ref{subsec:TypeA},
the family of threefolds $X_{2n+3}$, $n\geq1$, form a tower of $\mathbb{G}_{a}$-equivariant
affine modifications of the initial one $X_{5}$. 
\end{rem}

\bibliographystyle{amsplain}

\providecommand{\bysame}{\leavevmode\hbox to3em{\hrulefill}\thinspace}
\providecommand{\MR}{\relax\ifhmode\unskip\space\fi MR }
\providecommand{\MRhref}[2]{%
  \href{http://www.ams.org/mathscinet-getitem?mr=#1}{#2}
}
\providecommand{\href}[2]{#2}
\begin{thebibliography}{}

\end{thebibliography}


\begin{thebibliography}{99} 

\bibitem{DuPhD04} A. Dubouloz, \emph{Sur une classe de sch\'emas avec actions de fibr\'es en droites}, Ph.D Thesis, Universit\'e Joseph-Fourier-Grenoble I, \url{https://tel.archives-ouvertes.fr/tel-00007733/}, 2004.

\bibitem{DuTG05} A. Dubouloz \emph{Danielewski-Fieseler surfaces}, Transformation Groups vol. 10, no. 2, (2005), 139-162.

\bibitem{Du05} A. Dubouloz, \emph{Quelques remarques sur la notion de modification affine}, math.AG/0503142, (2005).

\bibitem{DF14} A. Dubouloz and D. R. Finston, \emph{ On exotic affine 3-spheres}, J. Algebraic Geom. 23 (2014), no. 3, 445-469.

\bibitem{Du15} A. Dubouloz, \emph{Complements of hyperplane sub-bundles in projective spaces bundles over $\mathbb{P}^1$}, Math. Ann. 361 (2015), no 1-2, 259-273.

\bibitem{EGA1} A. Grothendieck, \emph{\'El\'ements de G\'eom\'etrie Alg\'ebrique, I}, Publ.  Math. IHES, 4, 1960.

\bibitem{EGA3} A. Grothendieck, \emph{\'El\'ements de G\'eom\'etrie Alg\'ebrique, III}, Publ.  Math. IHES, 11 and 17, 1961 and 1963.

\bibitem{Fie94} K-H. Fieseler, \emph{On complex affine surfaces with $\mathbb{C}_+$-action}, Comment. Math. Helv. 69 (1994), no. 1, 5-27.

\bibitem{GMM12} R.V. Gurjar, K. Masuda and M. Miyanishi, \emph{$\mathbb{A}^1$-fibrations on affine threefolds},  Journal of Pure and Applied Algebra Volume 216, Issue 2 (2012), 296-313.

\bibitem{He15} I. Hed\'en, \emph{Affine extensions of principal additive bundles over a punctured surface}, Transform. Groups 21 (2),(2016), 427-449.

\bibitem{KaZa99} S. Kaliman and M. Zaidenberg, \emph{Affine modifications and affine hypersurfaces with a very transitive automorphism group}, Transform. Groups 4 (1999), no. 1, 53-95.

\bibitem{Ke80} G. Kempf, \emph{Some elementary proofs of basic theorems in the cohomology of quasi-coherent sheaves}, Rocky Mountain J. Maths. Volume 10, Number 3, 1980, 637-646.

\bibitem{Na59} M. Nagata, \emph{Lectures on the Fourteenth Problem of Hilbert}, Lecture Notes, Tata Institute, Bombay, 31, (1959). 

\bibitem{StackP} The {Stacks Project Authors}, \emph{Stacks Project}, \url{http://stacks.math.columbia.edu}, 2017.

\end{thebibliography}

\end{document}